\documentclass[12pt,british]{article}
\usepackage{color}
\usepackage{babel}
\usepackage{amsmath,amsthm}
\usepackage{amsfonts}
\usepackage{amssymb}
\usepackage{graphicx}
\usepackage[utf8]{inputenc}
\usepackage{multicol}
\usepackage[all]{xy}
\usepackage{bbm,enumerate}
\usepackage{geometry}
\usepackage[sharp]{easylist}

\newtheorem{theorem}{Theorem}

\newtheorem{corollary}[theorem]{Corollary}

\newtheorem{lemma}[theorem]{Lemma}
\newtheorem{proposition}[theorem]{Proposition}

\theoremstyle{definition}

\newtheorem{remark}[theorem]{Remark}

\newcommand{\z}{\mathbb{Z}}

\newcommand{\ztwo}{\mathbb{Z}_2}

\newcommand{\id}{\mathbbm{1}}
\newcommand{\free}{\mathcal{F}}
\renewcommand{\hom}{{\rm Hom}}

\usepackage[
bookmarks=true,
breaklinks=true,
bookmarksnumbered = true,
colorlinks= true,
urlcolor= green,
anchorcolor = yellow,
citecolor=blue,
]{hyperref}                   

\renewcommand{\l}{\left\langle}

\renewcommand{\r}{\right\rangle}

\begin{document}

\title{The Borsuk-Ulam property for homotopy classes on bundles, parametrized braids groups and applications for surfaces bundles}
\author{DACIBERG LIMA GON\c{C}ALVES
~\footnote{Departamento de Matemática, IME, Universidade de São Paulo, Rua do Matão, 1010 CEP: 05508-090, São Paulo, SP, Brazil. 
e-mail: \texttt{dlgoncal@ime.usp.br}}
\and
VINICIUS CASTELUBER LAASS
~\footnote{Universidade Federal da Bahia, IME, Departamento de Matemática, Av.\ Milton Santos, S/N Ondina CEP: 40170-110, Salvador, BA, Brazil. 
e-mail: \texttt{vinicius.laass@ufba.br}} 
\and
WESLEM LIBERATO SILVA
~\footnote{Departamento de Ciências Exatas e Tecnológicas, Universidade Estadual de Santa Cruz, Rodovia Jorge Amado, Km 16, Bairro Salobrinho, CEP 45662-900, Ilhéus-BA, Brazil.
e-mail: \texttt{wlsilva@uesc.br}}
}
%
\maketitle


		                \begin{abstract}

Let $M$ and $N$ be fiber bundles over the same base $B$, where $M$ is endowed with a free involution $\tau$ over $B$. A homotopy class $\delta \in [M,N]_{B}$ (over $B$) is said to have the Borsuk--Ulam property with respect to $\tau$ if for every fiber-preserving map $f\colon M \to N$ over $B$ which represents $\delta$ there exists a point $x \in M$ such that $f(\tau(x)) = f(x)$. In the cases that $B$ is a $K(\pi ,1)$-space and the fibers of the projections $M \to B$ and $N \to B$ are $K(\pi,1)$ closed surfaces $S_M$ and $S_N$, respectively, we show that the problem of decide if a homotopy class of a fiber-preserving map $f\colon M \to N$ over $B$ has the Borsuk-Ulam property is equivalent of an algebraic problem involving the fundamental groups of $M$, the orbit space of $M$ by $\tau$ and a type of generalized braid groups of $N$ that we call \emph{ parametrized braid groups}. As an application, we determine the homotopy classes of self fiber-preserving maps of some 2-torus bundles over $\mathbb{S}^1$
that satisfy the Borsuk--Ulam property with respect to certain involutions $\tau$ over $\mathbb{S}^1$.

\begingroup
\renewcommand{\thefootnote}{}
\endgroup

                \end{abstract}

\maketitle

\noindent


\section{Introduction}\label{introduction}
In the early twentieth century, St.~Ulam conjectured that if $f\colon\thinspace \mathbb{S}^n \to \mathbb{R}^n $ is a continuous map, there exists $ x \in \mathbb{S}^n$ such that $ f ( A (x)) = f (x) $, where $ A\colon\thinspace  \mathbb{S}^n \to \mathbb{S}^n $ is the antipodal map. The confirmation of this result by K.~Borsuk in 1933~\cite{Bor}, known as the Borsuk-Ulam theorem, was the beginning of what we now refer to as {\it Borsuk-Ulam type theorems} or the {\it Borsuk-Ulam property}. More information about the history and some applications of the Borsuk-Ulam theorem may be found in~\cite{Mato}, for example.

The St.~Ulam conjectured suggests a quite   general  question,
which  has recently been  studied,  which consist to substitute $\mathbb{S}^n$ and $\mathbb{R}^n$ by other spaces,  replace the antipodal map by a free involution (see \cite{BaGoHa,BarGonVen,GonSanSil}, for example),  where we ask:
does every continuous map collapse an orbit of the involution? More precisely, given topological spaces $M$ and $N$  such that $M$ admits a free involution $\tau$, we say that the triple {\it $(M , \tau ; N)$ has the Borsuk-Ulam property} if for every continuous map $f\colon M \to N$, there exists a point $x \in M$ such that $f(\tau(x)) = f(x)$. We emphasize that in the work \cite{GonGua}, the authors used surface braid groups to study this new type of the Borsuk-Ulam theorem for maps between closed surfaces.

More recently, another natural and more refined problem, that generalizes the one above, has been considered, 
namely the question of the classification of the free homotopy classes of maps between $M$ and $N$, where $M$ is equipped with a free $\z_2$-action for which every representative map in the homotopy class  collapses an orbit of the action. This problem was solved in three recent papers in the cases where $M$ and $N$ are compact surfaces without boundary of Euler characteristic zero~\cite{GonGuaLaa,GonGuaLaa2,GonGuaLaa3}. The results depend on the choice of free involution on $M$, and once more, braid theory plays an important r\^ole in the solution of the problem.

A natural generalization of the previous works is to study the Borsuk-Ulam problem in the setting of fiber bundles. More specifically, given $S_M \to M \stackrel{p}{\to} B$ and $S_N \to M \stackrel{q}{\to} B$ be fiber bundles over $B$, where $M$ is equipped with a free involution $\tau$ over $B$, and a fiber-preserving map $f: M \to N$  over $B$, we say that \emph{the homotopy class $\delta = [f]$ over $B$ has the Borsuk-Ulam property with respect to $\tau$ if every representative map of $\delta$ (which is a map homotopic over $B$ to $f$) is not injective in some orbit of $\tau$}. The focus of this article is to develop a tool to study this problem in the case that $M$ and $N$ are CW-complex, $B$ is $K(\pi,1)$ paracompact space and the fibers of $p$ and $q$ are $K(\pi,1)$ closed surfaces (see Theorem~\ref{th:mainborsuk_braid}).

%

Besides the introduction, this paper consists of three sections. In Section 2 we study the homotopy classes between fibrations. We prove an important relation between the set of theses classes and classes of homomorphisms  between  the fundamental groups of the spaces involved (see Theorem~\ref{th:set_homotopy}). Also, we give an algebraic criterion to decide when two given fiber-preserving maps over the same base are homotopic (see Theorem~\ref{th:criterion_homotopy}). In Section 3 we introduced the pure and the full parametrized braids groups in two strings of a fiber bundle. These groups are used to prove the main result of this paper which is Theorem~\ref{th:mainborsuk_braid}. In Section 4 we solved the problem of 
classify the homotopy classes of maps over $\mathbb{S}^1$ with respect to the Borsuk-Ulam property of some  2-torus bundle over $\mathbb{S}^{1}$ equipped with free involutions over $\mathbb{S}^1$ (see Theorems~\ref{main-result-1} and \ref{th:tau}). The others cases of $2$-torus bundle over $\mathbb{S}^1$ is the subject of work in progress.


\section{The set of homotopy classes of maps of a fibration}
 
Let $p\colon M \to B$ and $q\colon N \to B$ be fibrations, where $B$ is pathwise connected. Recall that a map $f\colon M \to N$ is said over $B$ if $q \circ f = p$. We say that two maps $f_0,f_1\colon M \to N$ are homotopic over $B$ if they are homotopic by a continuous family of maps $f_t\colon M \to N$, $t \in I = [ 0,1 ]$, over $B$. We denote the set of homotopy class of maps from $M$ to $N$ over $B$ by $[M,N]_B$. If we consider base points $m_1 \in M$ and $n_1 \in N$ such that $p(m_1) = q(n_1)$ and consider homotopies $f_t\colon (M,m_1) \to (N,n_1)$ over $B$ we have in an analogous way the set of pointed homotopy class over $B$ which we denote by $[M,m_1;N,n_1]_B$. Given an element $\alpha \in [ M ,m_1; N ,n_1 ]_B$, we may thus associate a free homotopy class $\alpha_\free\in [ M , N ]_B$  as follows:  choose a representative map $f\colon ( M , m_1 ) \rightarrow (N, n_1)$ of $\alpha$, and then  take the free homotopy class $\alpha_\free$ that contains  the map $f\colon M \rightarrow N$. So we obtain the following well-defined function:


\begin{equation}\label{defLambda}
	\begin{array}{rccl}
		\Lambda\colon & [ M , m_1 ; N , n_1 ]_{B} & \longrightarrow & [ M , N ]_{B} \\
		& \alpha = [f] & \longmapsto & \alpha_\free : = [f].
	\end{array}
\end{equation}

We denote by $\hom( \pi_1 (M,m_1), \pi_1 (N,n_1))_B$ the set whose elements are the homomorphisms $\psi\colon $ $ \pi_1 (M,m_1) \to \pi_1 (N,n_1) $ such that $  q_\# \circ \psi = p_\#$. Given an element $\alpha \in [ M , m_1 ; N , n_1]_{B}$, we may thus associate a homomorphism $\alpha_\# \in \hom (  \pi_1 ( M , m_1) , \pi_1 ( N , n_1))_B$ by choosing a representative map $f\colon ( M , m_1) \longrightarrow ( N , n_1)$ of $\alpha$, and by taking the induced homomorphism. It is easy  to see that the following function is well-defined:

\begin{equation}\label{defGamma}
	\begin{array}{rccl}
		\Gamma\colon & [ M , m_1 ; N , n_1 ]_{B} & \longrightarrow & \hom( \pi_1 (M,m_1), \pi_1 (N,n_1))_B \\
		& \alpha = [f] & \longmapsto & \alpha_\# : = f_\#.
	\end{array}
\end{equation}

Let $S_N = q^{-1}(q(n_1))$ be the  fiber of $q\colon N \to B$ ober $n_1$ and consider the inclusion $\iota\colon S_N \to N$. Two homomorphisms $h_1, h_2 \in \hom ( \pi_1 (M , m_1 ), \pi_1 (N , n_1 ))_B$ are said to be \emph{equivalent}, written $h_1 \sim_{S_N} h_2$, if there exists $\omega \in \iota_\# (\pi_1 (S_N , n_1))$ such that $h_2 ( \upsilon ) = \omega h_1 ( \upsilon) \omega^{-1}$ for all $\upsilon \in \pi_1 (M , m_1)$. It is straightforward to see that $\sim_{S_N}$ is an equivalence relation. The associated canonical projection shall be denoted as follows:
\begin{equation}\label{defUpsilon}
	\begin{array}{rccl}
		
		\Upsilon\colon & \hom ( \pi_1 (M , m_1), \pi_1 (N , n_1))_B & \longrightarrow & \dfrac{\hom ( \pi_1 (M , m_1), 
			\pi_1 (N , n_1))_B}{\sim_{S_N}} .
		
	\end{array}
\end{equation}

\begin{remark}
Note that if  $h_1$ is an element of $\hom ( \pi_1 (M , m_1 ), \pi_1 (N , n_1 ))_B$ and $h_2\colon \pi_1 (M,m_1) \to \pi_1(N,n_1)$ is a homomorphism such that $h_2 ( \upsilon ) = \omega h_1 ( \upsilon) \omega^{-1}$ for some $\omega \in \iota_\# ( \pi_1 (S_N , n_1))$ and all $\upsilon \in \pi_1 (M , m_1)$, then $h_2$ also belongs to $\hom ( \pi_1 (M , m_1 ), \pi_1 (N , n_1 ))_B$.
\end{remark}

The following result connects the previously defined objects and it is analogous to \cite[Theorem~4]{GonGuaLaa}.

\begin{theorem}\label{th:set_homotopy}
Let $p\colon M \to B$ and $q\colon N \to B$ be fibrations, $m_1 \in M$ and $n_1$ be base points such that $p(m_1)=q(n_1)$, and $S_M = p^{-1}(m_1)$ and $S_N = q^{-1}(n_1)$ be the fibers of $p$ and $q$, respectively. We assume that all spaces are pathwise connected $CW$-complex and the spaces $S_M$, $S_N$ and $B$ are $K(\pi,1)$. Then the function $\Gamma$ defined in~(\ref{defGamma}) is a bijection and the functions $\Lambda$ e $\Upsilon$ defined in~(\ref{defLambda}) and~(\ref{defUpsilon}), resp., are surjective. Moreover, there exists a bijective function $\Delta\colon [M,N]_B \to \dfrac{\hom ( \pi_1 (M , m_1), 
	\pi_1 (N , n_1))_B}{\sim_{S_N}}$ such that the following diagram is commutative:
		\begin{equation}\label{diag_homotopy}\begin{gathered}\xymatrix{
			[M , m_1 ; N , n_1 ]_B \ar[r]^-{\Gamma} \ar[d]_-{\Lambda} & \hom ( \pi_1 (M , m_1), \pi_1 (N , n_1))_B \ar[d]^-{\Upsilon} \\
			[M , N ]_B \ar[r]^-{\Delta} & \dfrac{\hom ( \pi_1 (M , m_1), \pi_1 (N , n_1))_B}{\sim_{S_N}}.
}\end{gathered}\end{equation}
\end{theorem}

Before to prove the above result we establish some auxiliary lemmas. In all of them, we assume the hypothesis of 
Theorem~\ref{th:set_homotopy}. Note that from the long exact  sequence associated to the fibrations we deduce that the spaces $M$ and $N$ are $K(\pi,1)$.

\begin{lemma}\label{lemma1}
Let $M \times_B N = \{ (x,y) \in M \times N \, | \, p(x) = q(y) \}$ be the pull-back of the two fibre maps $p\colon M \to B$ and $q\colon N\to B$, and let $i\colon M \times_B N \to M \times N$ be the inclusion. Then $M \times_B N$ has the homotopy type of a $K(\pi,1)$ $CW$-complex and the induced homomorphism  $i_\#\colon \pi_1 ( M \times_B N , (m_1,n_1)  ) \to \pi_1 (M , m_1 ) \times \pi_1 (N,n_1)$ is injective.
\end{lemma}

\begin{proof}
The projection onto in the first coordinate $p_1\colon M \times_B N \to M $ is a fibration whose typical fiber is $S_N$. Since we have that $M$ and $S_N$ are $CW$-complex, then $M \times_B N$ has the homotopy type of a $CW$-complex by \cite[Theorem~5.4.2]{FriPic}. Moreover, using the long exact 
sequence of homotopy groups, we deduce that $\pi_i( M \times_B N , (m_1,n_1) )$ is trivial for $i \neq 1$. Note that the following diagram is commutative
\begin{equation*}\begin{gathered}\xymatrix{ 
S_N \ar[r]^-{i_1} \ar[d]^-{\iota} & M \times_B N \ar[r]^-{p_1} \ar[d]^-{i}  & M \ar[d]^-{{\rm Id}} \\
N \ar[r]^-{i_2} & M \times N \ar[r]^-{p^\prime_1} & M,
}\end{gathered}\end{equation*}
where $i_1(x) = ( m_1 , x  )$, $i_2(x) = ( m_1 , x  )$ and $p^\prime_1$ is the projection onto the first coordinate. Take the fundamental group in the above diagram, we obtain the following commutative diagram, where the horizontal lines are exact sequences:
\begin{equation*}
\begin{gathered}\xymatrix{ 
1 \ar[r] & \pi_1(S_{N},n_1) \ar[r]^-{(i_1)_\#} \ar[d]^-{\iota_{\#}} & \pi_1(M \times_{B} N, (m_1,n_1)) \ar[r]^-{{p_1}_{\#}} \ar[d]^-{i_\#} &   \pi_{1}(M,m_1)  \ar[r] \ar[d]^-{{\rm Id}} & 1  	\\
1 \ar[r] & \pi_1(N,n_1)  \ar[r]^-{(i_2)_\#} & 
\pi_1(M \times N,(m_1,n_1)) \ar[r]^-{{p^\prime_1}_{\#}} &   
\pi_{1}(M,m_1)  \ar[r] & 1  .	\\
}	\end{gathered}\end{equation*}
Since $\iota_\#$ and ${\rm Id}$ are injective, then $i_\#$ is also.
\end{proof}

\begin{lemma}\label{lemma2}
Let $f,g\colon (M,m_1) \to (N,n_1)$ be maps over $B$ such that the induced homomorphisms $f_\# , g_\#\colon \pi_1 (M,m_1) \to (N,n_1)$ are equal. Then $f$ and $g$ are pointed homotopic over $B$.
\end{lemma}

\begin{proof}
Let $( 1 \times f) , (1 \times g)\colon (M,m_1) \to (M\times_B N ,(m_1,n_1))$ be defined by $(1 \times f)(x) = (x, f(x))$ and $(1 \times g)(x) = (x, g(x))$ and let $i\colon M \times_B N \to M \times N$ be the inclusion. By hypothesis, we have the equality $(i \circ (1 \times f))_\# = (i \circ (1 \times g))_\#$ on the level of fundamental groups. Since $i_\#$ is a monomorphism by Lemma~\ref{lemma1}, then $(1 \times f)_\# = (1 \times g)_\#$. Since $M \times_B N$ has the homotopy type of $K(\pi,1)$ $CW$-complex by Lemma~\ref{lemma1}, follows by \cite[Chapter~V, Theorem~4.3]{White} that there exists a homotopy $H\colon (M \times I,m_1 \times I)  \to (M\times_B N ,(m_1,n_1))$ such that $H(x,0) = (x,f(x))$ and $H(x,1) = (x,g(x))$. Let $p_2\colon (M\times_B N ,(m_1,n_1)) \to (N,n_1)$ be the projection onto the second coordinate and let $H^\prime = p_2 \circ H\colon (M \times I,m_1\times I) \to (N,n_1)$. Then $H^\prime$ is a homotopy over $B$ between $f$ and $g$.
\end{proof}

\begin{lemma}\label{lemma3}
Let $\psi \in \hom (\pi_1(M,m_1) , \pi_1(N,n_1))_B$. Then there exists a map $f\colon (M,m_1) \to (N,n_1)$ over $B$ such that $f_\# = \psi$.
\end{lemma}

\begin{proof}
Since all spaces are  $K(\pi; 1)$, follows by~\cite[Chapter V, Theorem 4.3]{White} that there exists a map $g\colon (M,m_1) \to (N,n_1)$ such that $g_{\#}=\psi$ and also there exists a homotopy $H\colon(M \times I, m_1 \times I) \longrightarrow (B, p(m_1) )$ such that $H(x,0)= q \circ g(x)$ and   $H(x,1)=p(x)$ for all $x \in M$. The map $G\colon(m_1 \times [0,1] \cup M \times 0, m_1 \times [0,1]) \longrightarrow (B, p(m_1))$ defined by $G(x,0)=g(x)$ and $ G(m_1 \times I)=n_1$ makes the square of the following diagram
$$\xymatrix {(m_1 \times I \cup M \times 0, m_1 \times I) 
	\ar[r]^-{G} \ar[d]_{i}    &   (N, n_1) \ar[d]^{q} \\
	( M \times I, m_1 \times I)  \ar@{-->}[ru]|{L} \ar[r]_-{H}    &   (B, p(m_1)) \\ }$$
commutative. Since $q\colon(N, n_1 ) \longrightarrow (B, p(m_1) )$ is a fibration it follows that there exists $L\colon( M \times I, m_1\times I) \longrightarrow (N ,n_1)$ such that the above diagram is commutative. Let $f(x)\colon (M ,m_1) \longrightarrow (N, n_1)$ defined by $f(x)=L(x,1)$. Finally, we have that $q \circ f(x) = q \circ L(x,1)= H(x,1) = p(x)$ for all $x \in M$ and thus $f$ is over $B$. We have that $ f_{\#} = L(\ast,0)_{\#} = g_{\#} = \psi$.
\end{proof}

\begin{lemma}\label{lemma4}
Given a map $f\colon M \to N$ over $B$, there exists a pointed map $g\colon (M,m_1) \to (N,n_1)$ over $B$ such that $f$ and $g$ are homotopic over $B$.
\end{lemma}

\begin{proof}
Let $n_2 = f(m_1) \in N$ and $\lambda\colon I \to N$ a path  such that $\lambda (0) = n_2$ and $\lambda(1) = n_1$. We define $F\colon  m_1 \times I \cup M \times 0 \to N$ by $F(x,0) = f(x)$ and $F(m_1,t)=\lambda(t)$. By \cite[Theorem~1.1.2 item~(i)]{FriPic}, there exists a homotopy $\tilde{F}\colon M \times I \to N$ which extends $F$ and such that $q \circ \tilde{F}(x,t) = p(x)$ for all $x \in M$ and $t \in I$. Let $g\colon M \to N$ be defined by $g(x) = \tilde{F}(x,1)$. Thus $g$ is the desired map and we leave to the reader to check the details.
\end{proof}

\begin{lemma}\label{lemma5}
Let $f,g\colon M \to N$ be maps over $B$ such that $f(m_1) = g(m_1) = n_1$. Then $f$ and $g$ are homotopics over $B$ if, and only if, $f_\# \sim_{S_N} g_\#$.
\end{lemma}

\begin{proof}
Suppose first that there exists a homotopy $H\colon M \times I \to N$ over $B$ such that $H(x,0) = f(x)$ and $H(x,1) = g(x)$. Let $\alpha_1$ be the loop on $N$ defined by $\alpha_1(t) = H(m_1,t)$ and let $\omega_1 = [ \alpha_1 ] \in \pi_1(N,n_1)$. Note that $\omega_1 \in \iota_\# (\pi_1(S_N,n_1))$ and by construction we have $f_\# ( v ) = \omega_1 g_\# (v) \omega_1^{-1}$ for all $v \in \pi_1(M,m1)$ which ends the first part of the proof.
Now, suppose that $f_\# \sim_{S_N} g_\#$. Let $\omega_2 = [ \alpha_2 ] \in \iota_\# (\pi_1(S_N,n_1)) $ such that $f_\# ( v ) = \omega_2 g_\# (v) \omega_2^{-1}$ for all $v \in \pi_1(M,m_1)$. We define $H\colon m_1 \times I \cup M \times 0 \to N$ by $H(m,t) = \alpha(t)$ and $H(x,0) = f(x)$. By \cite[Theorem~1.1.2 item~(i)]{FriPic}, there exists a homotopy $\tilde{H}\colon M \times I \to N$ that extends $H$. By construction, $H(m_1,1) = n_1$ and the maps $f$ and $\tilde{H}(\ast,1)$ are homotopics over $B$. Note that on the level of fundamental groups we have $\tilde{H}(\ast,1)_\# = \omega_2^{-1} f_\# \omega_2 = g_\#$. So $\tilde{H}(\ast,1)$ and $g$ are homotopics over $B$ by Lemma~\ref{lemma2} which ends the second part of the proof.
\end{proof}

\begin{proof}[Proof of Theorem~\ref{th:set_homotopy}.]
Follows from Lemma~\ref{lemma2} that $\Gamma$ is injective and follows from Lemma~\ref{lemma3} that $\Gamma$ is surjective. Lemma~\ref{lemma4} assures that $\Lambda$ is surjective. It is clear that $\Upsilon$ is surjective. Given $\beta \in [M,N]_B$, let $\alpha \in [M,m_1,N,n_1]_B$ such that $\Lambda(\alpha) = \beta$. The {\it if} part  of Lemma~\ref{lemma5} assures that the correspondence $\Delta(\beta) = \Upsilon \circ \Gamma (\alpha)$ defines a surjective function $\Delta\colon [M,N]_B \to \dfrac{\hom ( \pi_1 (M , m_1), \pi_1 (N , n_1))_B}{\sim_{S_N}}$. The {\it only if} part of Lemma~\ref{lemma5} assures that $\Delta$ is injective. By construction the diagram~(\ref{diag_homotopy}) is commutative.
\end{proof}

The Theorem~\ref{th:set_homotopy} give us an algebraic description of the set of homotopy classes between $K(\pi,1)$ fibrations. To end this section, we enunciate and prove the following algebraic criterion to decide when two maps between fibrations are homotopic.

\begin{theorem}\label{th:criterion_homotopy}
Suppose the same hypothesis of Theorem~\ref{th:set_homotopy}. Let $f,g\colon M \to N$ be maps over $B$. Given a point $m_1 \in M$ and a path $\gamma\colon I \to S_N$ from $f(m_1)$ to $g(m_1)$, let $\gamma_\# \colon \pi_1 (N,g(m_1)) \to \pi_1 (N,f(m_1))$ be the canonical isomorphism determined by
 $\gamma$ that changes the basepoints of the fundamental group of $N$. Then $f$ and $g$ are homotopic over $B$ if, and only if, $ f_\# \sim_{S_N} \gamma_\# \circ g_\#$. 
\end{theorem} 

\begin{proof}
Suppose first that $f$ and $g$ are homotopic over $B$. Let $h\colon M \times \{ 0 \} \cup \{ m_1 \} \times I \to N$ be the map defined by $h(x,0)=f(x)$ and $h(m_1,t)=\gamma(t)$. Since $q\colon N \to B$ is a fibration,  there exists a homotopy $H\colon M \times I \to N$ over $B$ that extends $h$ by \cite{AllFad}. Let $g_1\colon M \to N$ be the map defined by $g_1(x) = H(x,1)$. Then $f$ and $g_1$ are homotopic over $B$, $g_1(m_1)=g(m_1)$ and it follows in a straightforward manner that $f_\# = \gamma_\# \circ (g_1)_\#$. Thus $g_1$ and $g$ are homotopic over $B$ and therefore $(g_1)_\# \sim_{S_N} g_\#$ by Theorem~\ref{th:set_homotopy}. This implies that $f_\# \sim_{S_N} \gamma_\# \circ g_\#$. Now, let us prove the reciprocal. By Theorem~\ref{th:set_homotopy} there exists a map $f_1\colon(M,m_1) \to (N,f(m_1))$ over $B$ such that $(f_1)_\# = \gamma_\# \circ g_\#$. Again by Theorem~\ref{th:set_homotopy} and the assumption, it follows that $f$ and $f_1$ are homotopic over $B$. Using similar arguments that was used in the first part of the proof, we can obtain a map $g_2\colon M \to N$ over $B$ such that $f_1$ and $g_2$ are homotopic over $B$, $g_2(m_1) = g(m_1)$ and $(f_1)_\# = \gamma_\# \circ (g_2)_\#$. So $(g_2)_\# = g_\#$ and once more by Theorem~\ref{th:set_homotopy} we have that $g_2$ and $g$ are homotopic over $B$. By transitivity it follows that $f$ and $g$ are homotopic over $B$. 
\end{proof}


\section{The Borsuk-Ulam property and braid groups}

Let $p\colon M \to B$ be a fibration where $M$ is a connected CW-complex and $B$ is pathwise connected. Suppose that $M$ is equipped with a cellular free involution $\tau$ over $B$. Denote by  $M_\tau$ the corresponding \emph{orbit space}, and let $\pi_\tau\colon M \to M_\tau$ denote the associated double covering. Note that the map $p$ induces a map $\overline{p_\tau}\colon M_\tau \to B$ which satisfies $p =  \overline{p_\tau} \circ \pi_\tau$. We have the following commutative diagram
\begin{equation}\label{seq_tau}\begin{gathered}\xymatrix{
			1 \ar[r] & \pi_1 ( M , m_1) \ar[rr]^-{(\pi_\tau)_\#} \ar[rd]_-{p_\#} 
			& & \pi_1 ( M_\tau , \pi_\tau(m_1)) \ar[r]^-{\theta_\tau} \ar[ld]^-{\overline{p_\tau}_\#} 
			&  \ztwo \ar[r] & 1 \\
			& &   \pi_1( B , p(m_1) ), & & &
}\end{gathered}\end{equation}
where we identify the group $\dfrac{\pi_1 ( M_\tau , p_\tau(m_1))}{ ( \pi_\tau )_\# ( \pi_1 ( M , m_1) )}$ with $\ztwo = \{ \overline{0} , \overline{1} \} $, $\theta_{\tau}$ is the natural projection onto the quotient and the horizontal line is an exact sequence. About the map $\overline{p_\tau}$ we have the following result.

\begin{lemma}\label{lem:quo_fibration}
The map $\overline{p_\tau}\colon M_\tau \to B$ is a fibration.
\end{lemma}

\begin{proof}
Recall that fibration means Hurewics fibration and in the category of CW-complex a Serre fibration is also a Hurewics fibration (see \cite{Rob,SteWes}). Since $M_\tau$ is a CW-complex, let us prove that $\overline{p_\tau}$ is a Serre fibration. So suppose that we have the following commutative diagram:
$$\xymatrix{
I^n \times \{0 \} \ar[r]^-{h} \ar@{^{(}->}[d] & M_\tau \ar[d]^-{\overline{p_\tau}} \\
I^n \times I \ar[r]^-H & B,
}$$
where $I$ denotes the unit interval. Since $I^n$ is simply-connected, there exists a map $\tilde{h}\colon I^n \times \{ 0 \} \to M$ such that $\pi_\tau \circ \tilde{h} = h$ and since $p\colon M \to B$ is a fibration we obtain the following commutative diagram:
$$\xymatrix{
& M \ar[d]^-{\pi_\tau} \ar@/^1.0cm/[dd]^-{p} \\
I^n \times \{0 \} \ar[r]^-{h} \ar@{^{(}->}[d] \ar[ru]^-{\tilde{h}} & M_\tau \ar[d]^-{\overline{p_\tau}} \\
I^n \times I \ar[r]^-H  \ar@/^2.6cm/[ruu]_-{ \tilde{H} } & B.
}$$
We define $\bar{H}=\pi_\tau \circ \tilde{H}\colon I^n \times I \to M_\tau$. From the last diagram follows that $\overline{p_\tau} \circ \bar{H} = H$ and $\bar{H}|_{I^n \times \{0 \}} = h$. Thus $\overline{p_\tau}$ is a Serre fibration.
\end{proof}

The following Lemma is analogous to \cite[Lemma 5]{GonGuaLaa} for fibrations.

\begin{lemma}\label{equiv_maps}
Suppose the same hypothesis of Theorem~\ref{th:set_homotopy} and also that $M$ and $N$ are equipped with cellular free involutions $\tau$ and $\tau_1$, respectively, over $B$. Given a pointed homotopy class $\alpha \in [M , m_1 ; N , n_1]_{B}$, the following conditions are equivalent:
\begin{enumerate}
	
\item there exists a representative map  $f\colon ( M , m_1 ) \to ( N , n_1)$ of $\alpha$ such that $f$ is over $B$ and it is $(\tau, \tau_1)$-equivariant, i.e.\  $q \circ f = p$ and $f \circ \tau = \tau_1 \circ f$.
	
\item there exists a homomorphism  $\psi\colon \pi_1 ( M_\tau , \pi_\tau (m_1) ) \to \pi_1 ( N_{\tau_1} , \pi_{\tau_1 } (n_1) )$ such that the following diagram is commutative:
	
\begin{equation}\label{diag_equiv}\begin{gathered}\xymatrix{
\pi_1(M,m_1) \ar[rr]^-{ \alpha_\# } \ar[dd]_-{ (\pi_\tau)_\# } 
\ar[dr]^-{p_{\#}} & & \pi_1 ( N , n_1) \ar[dd]^-{ ( \pi_{\tau_1} )_\# } \ar[dl]_-{q_{\#}} \\
&   \pi_{1}(B, p(m_{1})) & \\
\pi_1( M_\tau, \pi_\tau( m_1 )) \ar@{.>}[rr]^-{\psi} \ar[rd]_-{\theta_\tau} \ar[ur]^-{ \overline{p_\tau}_\# } &     & \pi_1 ( N_{\tau_1}, \pi_{\tau_1} (n_1) ) \ar[ld]^-{ \theta_{\tau_1}} \ar[ul]_-{ \overline{q_{\tau_1}}_\# } \\
  & \ztwo ,& 
}\end{gathered}\end{equation}
where $\alpha_\#$ is as defined in~(\ref{defGamma}).		
		\end{enumerate}
		\end{lemma}

\begin{proof} 
$(1 \Rightarrow 2)$ Let $f\colon (M,m_1) \to (N,n_1)$ be an equivariant map over $B$ that represents $\alpha$. Then $q \circ f = p$ and $f \{ z , \tau(z) \} = \{ f(z) , \tau_1 ( f(z)) \}$ for all $z \in M$. Then the map $\overline{f}\colon M_\tau \to N_{\tau_1}$ given by $\overline{f}[z] = [f(z)]$ is well-defined and it satisfies $\pi_{\tau_1} \circ f = \overline{f} \circ \pi_\tau$. We already have $\overline{p_\tau} \circ \pi_\tau = p$ and $\overline{q_{\tau_1}} \circ \pi_{\tau_1} = q$. Since $\pi_\tau$ and $\pi_{\tau_1}$ are surjective maps, follows that $\overline{q_{\tau_1}} \circ \overline{f} = \overline{p_\tau}$. So, we have the following commutative diagram:
\begin{equation}\label{diag_equiv_aux}\begin{gathered}\xymatrix{
(M , m_1) \ar[rr]^-{f} \ar[dd]_-{ \pi_\tau } \ar[dr]^{p}  &   & (N , n_1) \ar[dd]^-{ \pi_{\tau_1} }  \ar[dl]_{q}  \\
& (B, p(m_{1})) &  \\
(M_\tau, \pi_\tau(m_1)) \ar[rr]^-{ \overline{f} } \ar[ur]^{\overline{p_\tau}}  &   & ( N_{\tau_1}, \pi_{\tau_1} (n_1) ) 
\ar[ul]_{\overline{q_{\tau_1}}}.
}\end{gathered}\end{equation}
Let $\psi\colon \pi_1 ( M_\tau, \pi_\tau ( m_1) ) \to \pi_1 ( N_{\tau_1} , \pi_{\tau_1} (n_1))$ be the induced homomorphism by $\overline{f}$. We obtain the upper square of the diagram~(\ref{diag_equiv}) by taking the fundamental group in the diagram~(\ref{diag_equiv_aux}). We obtain the lower triangle of the diagram~(\ref{diag_equiv}) by using standard covering space arguments and the fact that the map $f$ is equivariant.

		
\noindent $(2 \Rightarrow 1)$ Suppose that there exists a homomorphism $\psi\colon \pi_1 ( M_\tau , \pi_\tau(m_1) ) \to \pi_1 ( N_{\tau_1} , \pi_{\tau_1}(n_1) )$ for which diagram (\ref{diag_equiv}) is commutative. From the homotopy long exact sequence of the fibration we obtain that $N$ is $K(\pi,1)$ and therefore $N_{\tau_1}$ is also because $\pi_{\tau_1}$ is a covering map. Since $\overline{p_\tau}$ and $\overline{q_{\tau_1}}$  are fibrations by Lemma~\ref{lem:quo_fibration}, it follows from Theorem~\ref{th:set_homotopy} that there exists a map $\overline{f}_0\colon ( M_\tau , \pi_\tau (m_1) ) \to ( N_{\tau_1} , \pi_{\tau_1} (n_1) )$ such that $(\overline{f}_0)_\# = \psi$ and $(\overline{q_{\tau_1}} \circ \overline{f_0})_\# = \overline{p_\tau}_\#$. By the ladder  equality and Theorem~\ref{th:set_homotopy}, that there exists a homotopy between $\overline{q_{\tau_1}} \circ \overline{f_0}$ and $\overline{p_\tau}$ which is over $B$ and keeps the base points stationary. Using this homotopy and, once more that $\overline{q_{\tau_1}}$ is a fibration, 
 to change $\overline{f_0}$ by a homotopic map $\overline{f}$ such that $\overline{f}_\# = \psi$ and $\overline{q_{\tau_1}} \circ \overline{f}= \overline{p_\tau}$. Let $f\colon (M , m_1 ) \to ( N , n_1)$ be the lift of the map $\overline{f} \circ \pi_\tau$ for the covering $\pi_{\tau_1}$. We have  
$$q \circ f = 
\overline{q_{\tau_1}} \circ \pi_{\tau_{1}} \circ f =
\overline{q_{\tau_1}} \circ \overline{f} \circ \pi_{\tau}  =
\overline{p_\tau} \circ \pi_{\tau}  = p$$ 
and therefore $f$ is over $B$. So, we have the commutative diagram~(\ref{diag_equiv_aux}). Note that
$$
(\pi_{\tau_1})_\# \circ \alpha = \psi \circ (\pi_\tau)_\# = \overline{f}_\# \circ (\pi_\tau)_\# = (\pi_{\tau_1})_\# \circ f_\#.
$$
Since $(\pi_{\tau_1})_\#$ is a monomorphism, then $\alpha = f_\#$ which implies that $f \in \alpha$ by Theorem~\ref{th:set_homotopy}. We claim that $f$ is $(\tau, \tau_1)$-equivariant. To do so, note that for all $z \in M$, we have:
$$(\pi_{\tau_1} \circ f \circ \tau )( z )
= ( \overline{f} \circ \pi_\tau \circ \tau)( z )
= ( \overline{f} \circ \pi_\tau)( z )
= ( \pi_{\tau_1} \circ f ) ( z ).$$
Hence either $f( \tau ( z )) = f( z )$ or  $f( \tau ( z )) = \tau_1 (f ( z ))$. Let $\xi\colon [0,1] \to M$ be an arc from $m_1$ to $\tau(m_1)$. Then the loop $\gamma = \pi_\tau \circ \xi$ satisfies $\theta_\tau ( [ \gamma ]) = \overline{1}$. By (\ref{diag_equiv}) we have:
$$\theta_{\tau_1} ( \left[ \overline{f} \circ \gamma \right] ) = \theta_{\tau_1} \circ \overline{f}_\# ( [ \gamma ] )= 
\theta_\tau ( \left[ \gamma \right] ) = \overline{1},$$
and since $ f \circ \xi$ is a lift of the loop $\overline{f} \circ \gamma$ by the covering $\pi_{\tau_1}$, it is an arc that is not a loop. Therefore $f( m_1) = (f \circ \xi) (0) \neq (f \circ \xi) ( 1 ) = f ( \tau ( m_1))$, and so $f ( \tau ( m_1)) = \tau_1 ( f ( m_1))$. Using standard covering space arguments and the hypothesis that $S_{N}$ is $K(\pi,1)$ space which implies $N$ connected, it follows in a straightforward manner that the  equality  $f ( \tau (z)) = \tau_1 ( f (z))$ holds for all  $z\in M$. 
\end{proof}
	
Let $q\colon N \to B$ be a map such that $q^{-1}(b)$ is homeomorphic to space $S_N$ for all $b \in B$. Recall that $F_2 (S_N) = \{ (x,y) \in S_N \times S_N \, | \, x \neq y \}$ is the 2-ordenered configuration space of $S_N$ and the 2-unordered configuration space of $S_N$, which we denote by $D_2(S_N)$, is the quotient space of $F_2(S_N)$ by the equivalence relation generated by $(x,y) \sim (y,x)$. We consider the pull back $N \times_B N$ of $q\colon N \rightarrow B$ by $q\colon N \rightarrow B$ which means that $N \times_{B} N = \{  (x,y) \in N \times N \, ; q(x) = q(y) \} $. We define
$$F_2(N)_{B} = N \times_{B} N - \Delta,$$
where $\Delta$ is the diagonal in $N \times_{B} N$. Consider the composition $q \circ p_1\colon F_2(N) \to B$, where $p_1$ is the projection onto the first coordinate. Note that a preimage of point of $B$ by $q \circ p_1$ is, up to homeomorphism, equal to to $F_2(S_N)$. Let $\tau^{'}$ be the free involution on $N \times_B N -\Delta$ defined by  $\tau^{'}(x,y) = (x,y)$. We define
$$D_2(N)_{B} = \displaystyle\frac{N \times_{B} N-\Delta}{\tau^{'}},$$
be the quotient space of $N \times_{B} N-\Delta$ by the equivalence relation generated by $(x,y) \sim \tau^{'}(x,y)$. Note that $q \circ p_1$ induces a map $\overline{q \circ p_1}\colon D_2(N)_{B} \to B$ and a preimage of a point is, up to homeomorphism, equal to $D_2(S_N)$. Moreover, the following diagram is commutative:
\begin{equation}\label{diag-fiber-bundle-geo}\begin{gathered}\xymatrix{
F_2(S_N) \ar@{^{(}->}[r] \ar[d]^{{\pi_{\tau'}}_|}&  F_2(N)_{B}  \ar[r]^-{q \circ p_1} \ar[d]^-{\pi_{\tau'}}  & B  \ar[d]^-{\rm Id}\\
D_2(S_N) \ar@{^{(}->}[r] & D_2(N)_{B}  \ar[r]^-{\overline{q \circ p_1}} & B.
}\end{gathered}\end{equation}

We have the following result in the context of fiber bundles.

\begin{lemma}\label{lem:fiber-config}
If $q\colon N \to B$ is a fiber bundle with fiber $S_N$, then the maps $q \circ p_1$ and $\overline{q \circ p_1}$ of the diagram~(\ref{diag-fiber-bundle-geo}) are fiber bundles with fibers $F_2(S_N)$ and $D_2(S_N)$, respectively.
\end{lemma}

\begin{proof}
Let $b \in B$. Then, there exists a neighborhood $U$ of $b$ and a homeomorphism $\phi\colon U \times S_N \to q^{-1}(b)$ such that the following diagram is commutative:
$$\xymatrix{
U \times S_N \ar[r]^-{\phi} \ar[rd]_-{p_1}  & q^{-1}(b)  \ar@{^{(}->}[r] &  N \ar[dl]^-{q}  \\
 & B .&
}$$
We leave to the reader to show that the following maps are well-defined homeomorphisms:
$$\phi_1\colon U \times F_2(S_N) \to (q\circ p_1)^{-1}(b) \qquad \qquad \phi_1(z,(x,y)) = (\phi(z,x),\phi(z,y))$$
and
$$\phi_2\colon U \times D_2(S_N) \to (\overline{q\circ p_1})^{-1}(b) \qquad \qquad\phi_2(z,\overline{(x,y)}) = \overline{ (\phi(z,x),\phi(z,y)) }.$$
Moreover, the following diagrams are commutatives:

$$\xymatrix{
	U \times F_2(S_N) \ar[r]^-{\phi_1} \ar[rd]_-{p_1}  & (q \circ p)^{-1}(b)  \ar@{^{(}->}[r] &  F_2(N)_B \ar[dl]^-{q \circ p_1}  \\
	& B .&
}
\qquad
\xymatrix{
	U \times D_2(S_N) \ar[r]^-{\phi_2} \ar[rd]_-{p_1} &  (\overline{q \circ p})^{-1}(b)  \ar@{^{(}->}[r] & D_2( N )_B \ar[dl]^-{\overline{q \circ p_1}}  \\
	& B .&
}$$	
\end{proof}

\begin{corollary}\label{cor:fiberbundle_fibration}
With respect to Lemma~\ref{lem:fiber-config}, if $B$ is a paracompact space, then $q$, $q\circ p_1$ and $\overline{q \circ p_1}$ are also fibrations.
\end{corollary}

Suppose now that $S_N$ is a surface and let $(n_1,n_2) \in F_2(S_N)$. Recall that $ P_2 (S_N) = \pi_1 (F_2(S_N),(n_1,n_2))$ and $B_2 (S_N) = \pi_1 (D_2(S_N),\overline{(n_1,n_2)})$ are the $2$-pure braid group of $S_N$ and the $2$-full braid group of $S_N$, respectively. We denote
$$P_2(N)_{B} = \pi_1 (F_2(N)_{B},(n_1,n_2)) \text{ and } B_2(N)_{B} = \pi_1 (D_2(N)_{B},\overline{(n_1,n_2)})$$
which we call the \emph{parametrized $2$-pure braid group of $N$} and the \emph{parametrized $2$-full braid group of $N$}, respectively. The group $P_2(N)_{S^{1}}$ has been appeared in \cite{FadHus-81} in a study of fixed points of fiber-preserving maps.

Suppose now that $S_N$ is the Euclidean plane or a $K(\pi,1)$ closed surface, $N$ is connected CW-complex and $B$ is a $K(\pi,1)$ paracompact space. Then $F_2(S_N)$ is $K(\pi,1)$ by~\cite[Corollary 2.2]{FadNeu} and therefore $D_2(S_N)$ is also a $K(\pi,1)$. Taking the long exact sequence of homotopy in the diagram~(\ref{diag-fiber-bundle-geo}) we obtain the following algebraic commutative diagram, where the horizontal lines are exact sequences:

\begin{equation}\label{diag-fiber-bundle-induced-2}\begin{gathered}\xymatrix{ 
1 \ar[r] & P_2(S_N) \ar@{^{(}->}[r] \ar[d]^-{({\pi_{\tau'}}_|)_\#} & P_2(N)_{B} \ar[r]^-{(q\circ p_1)_\#} \ar[d]^-{(\pi_{\tau'})_\#}  & \pi_1(B) \ar[r] \ar[d]^-{\rm Id} & 1 \\
1\ar[r] & B_2(S_N) \ar@{^{(}->}[r] &  B_2(N)_{B} \ar[r]^-{\overline{q \circ p_1}_\#} & \pi_1(B) \ar[r] & 1 .
}\end{gathered}\end{equation}

The following result is a generalization of \cite[Theorem~7]{GonGuaLaa} for fibrations and fiber bundles.

\begin{theorem}\label{th:mainborsuk_braid}
Let $M$ and $N$ connected CW-complex and $B$ a $K(\pi,1)$ paracompact space, let $p\colon M \to B$ be a fibration and $q\colon N \to B$ be a fiber bundle whose fibers are $K(\pi,1)$ closed surfaces and suppose that $M$ is equipped with a cellular free involution $\tau$ over $B$. Given points $m_1 \in M_1$, $n_1 \in N$ and $(n_1,n_2) \in F_2(N)_B$ such that $p(m_1)=q(n_1)$, a pointed pointed homotopy class $\alpha \in [M,m_1;N,n_1]_{B}$ and a free homotopy class $\beta \in [M,N]_B$ such that $\alpha_\free = \beta$, the following conditions are equivalent:
\begin{enumerate}[(a)]

\item\label{th:mainborsuk_braid_a} $\alpha$ does not have the Borsuk-Ulam property with respect to $\tau$, i.\ e, that there exists a map $f: (M,m_1) \to (N,n_1)$ over $B$ that represents $\alpha$ which satisfies $f(\tau(z)) \neq f(z)$ for all $z \in M$.

\item\label{th:mainborsuk_braid_b} there exist homomorphisms $\varphi\colon \pi_1 ( M , m_1) \to P_2 (N)_{B}$ and $\psi\colon \pi_1 (M_\tau, \overline{m_1}) \to B_2 ( N )_{B} $ for which the following diagram is commutative:

\begin{equation}\label{diag_borsuk_braid}\begin{gathered}\xymatrix{
\pi_1 (M,m_1) \ar@{.>}[rr]^{\varphi} \ar[dd]_{(\pi_\tau)_\#} \ar@/^0.9cm/[rrrr]^{ \alpha_\# } \ar[dr]^-{p_{\#}} 
&& P_2(N)_{B} \ar[dd]^-{(\pi_{\tau^{'}})_{\#}} \ar[rr]^{(p_1)_\#} \ar[dl]_-{({q \circ p_{1}})_{\#}} && \pi_1(N,n_1)  \\
&  \pi_{1}(B,p(m_{1}))  &  & &  \\
\pi_1(M_\tau,\overline{m_1}) \ar@{.>}[rr]^{\psi} \ar[rd]_{\theta_\tau} \ar[ur]^-{ \overline{p_\tau}_\# } && B_2(N)_{B} \ar[ld]^{\theta_{\tau^{'}}} \ar[ul]_-{\overline{q \circ p_1}_\# } && \\
& \ztwo .& & &
}\end{gathered}\end{equation}
\item\label{th:mainborsuk_braid_c} $\beta$ does not have the Borsuk-Ulam property with respect to $\tau$.
\end{enumerate}	
\end{theorem}


\begin{proof}
	
Suppose first that condition~(\ref{th:mainborsuk_braid_a}) holds and let us to show that condition~(\ref{th:mainborsuk_braid_b}) holds. Then there exists a map $f\colon ( M , m_1) \to (N, n_1)$ over $B$ such that $f \in \alpha$ and $f( \tau (z)) \neq f(z)$ for all $z \in M$. Hence the map $F\colon (M,m_1) \to (F_2(N)_B,(n_1,n'_2))$ given by $F(z) = ( f( z) , f( \tau(z)))$ is well defined and over $B$, where $n'_2 = f(\tau(m_1))$. Note that $F$ is $(\tau, \tau_1)$-equivariant, and satisfies $p_1 \circ F = f$.  If $\varphi_1\colon  \pi_1 ( M , m_1) \to \pi_1( F_2(N)_B, (n_1,n'_2))$ is the homomorphism induced  by $F$ then $(p_1)_\# \circ \varphi_1 = (p_1)_\# \circ F_\# = ( p_1 \circ F )_\# = f_\# = \alpha_\#$. Moreover, by Theorem~\ref{th:set_homotopy}, Lemma~\ref{equiv_maps} and Corollary~\ref{cor:fiberbundle_fibration} applied for the pointed homotopy class $[F]$, there exists a homomorphism $\psi_1\colon  \pi_1 ( M_\tau , \overline{m_1}) \to \pi_1(D_2(N)_B,\overline{(n_1,n'_2}))$ which give us the following commutative diagram:
\begin{equation}\label{diag_aux_algebric}\begin{gathered}\xymatrix{
\pi_1 (M,m_1) \ar@{.>}[rr]^{\varphi_1} \ar[dd]_{(\pi_\tau)_\#} \ar@/^0.9cm/[rrrr]^{ \alpha_\# } \ar[dr]^-{p_{\#}} 
&& \pi_1(F_2(N)_{B},(n_1,n'_2)) \ar[dd]^-{(\pi_{\tau^{'}})_{\#}} \ar[rr]^{(p_1)_\#} \ar[dl]_-{({q \circ p_{1}})_{\#}} && \pi_1(N,n_1)  \\
&  \pi_{1}(B,p(m_{1}))  &  & &  \\
\pi_1(M_\tau,\overline{m_1}) \ar@{.>}[rr]^{\psi_1} \ar[rd]_{\theta_\tau} \ar[ur]^-{ \overline{p_\tau}_\# } && \pi_1(D_2(N)_{B}, \overline{(n_1,n_2'})) \ar[ld]^{\theta_{\tau^{'}}} \ar[ul]_-{\overline{q \circ p_1}_\# } && \\
& \ztwo .& & &
}\end{gathered}\end{equation}
Note that $(n_1,n_2)$ and $(n_1,n'_2)$ are in the same fiber of $q \circ p_1\colon F_2(N)_B \to B$ which is the pathwise connected space $F_2(S_N)$. So, there exists a path $\gamma$ in $F_2(N)_B$ whose image is cointaned in $F_2(S_N)$ such that $\gamma(0) =(n_1,n_2)$ and $\gamma(1) = (n_1,n'_2)$. Then the path $\overline{\gamma} = \pi_{\tau'} \circ \gamma$ in $D_2(N)_B$ satisfies $\overline{\gamma}(0) = \overline{(n_1,n_2)}$ and $\overline{\gamma}(1) = \overline{(n_1,n'_2})$. Let $\gamma_\#\colon \pi_1 (F_2(N)_B, (n_1,n'_2)) \to P_2(N)_B = \pi_1(F_2(N)_B,(n_1,n_2))$ and  ${\overline{\gamma}}_\#\colon \pi_1 (D_2(N)_B, \overline{(n_1,n'_2)} \to B_2(N)_B = \pi_1(D_2(N)_B,(n_1,n_2))$ be the isomorphisms induced by the paths $\gamma$ and $\overline{\gamma}$, respectively. Finnaly, let $\varphi = \gamma_\# \circ \varphi_1\colon \pi_1(M,m_1) \to P_2 (N)_B$ and $\psi = {\overline{\gamma}}_\# \circ \psi_1\colon \pi_1(M_\tau,m_1) \to B_2(N)_B$. With these homomorphisms we obtain the diagram~(\ref{diag_borsuk_braid}) from the diagram~(\ref{diag_aux_algebric}) and we leave to reader to check the details.

Conversely, suppose that condition~(\ref{th:mainborsuk_braid_b}) holds and let us show that $\alpha$ does not have the Borsuk-Ulam property with respect to $\tau$. By diagram~(\ref{diag-fiber-bundle-induced-2}), we have that $F_2(N)_{B}$ is $K ( \pi , 1)$ and so, Theorem~\ref{th:set_homotopy} and Lemma~\ref{equiv_maps} imply that the homomorphism $\varphi$ is induced by a $(\tau, \tau_1)$-equivariant map $F\colon  (M,m_1) \to (F_2(N)_B, (n_1,n_2))$ which is over $B$. Let $f, g\colon  M \to N$ be maps such that $F (z) = ( f( z) , g(z))$ for all $z \in M$. Since $F$ is equivariant, we have $f( \tau (z)) = g(z) \neq f(z)$ for all $z \in M$. Again, by~(\ref{diag_borsuk_braid}) and Theorem~\ref{th:set_homotopy} we have $\alpha = [f]$, and thus $\alpha$ does not have the Borsuk-Ulam property, and so~(\ref{th:mainborsuk_braid_a}) holds.

It is easy to see that condition~(\ref{th:mainborsuk_braid_a}) implies condition~(\ref{th:mainborsuk_braid_c}). To finish the proof, let us show the reciprocal. We already know that conditions~(\ref{th:mainborsuk_braid_a}) and~(\ref{th:mainborsuk_braid_b}) are equivalent and we will use this information. So suppose that $\beta$ does not have the Borsuk-Ulam which means that there exists a map $f\colon  M \to N$ over $B$ such that $\beta = [f]$ and $f( \tau ( z) ) \neq f(z)$ for all $z \in M$. Let $n'_1 = f(m_1)$. Then the pointed homotopy class $\alpha_1 = [f] \in [M,m_1;N,n'_1]_B$ does not have the Borsuk-Ulam property with respect to $\tau$. So, choosing $n'_2 \in N$ such that $(n'_1,n'_2) \in F_2(N)_B$, we have the commutative diagram~(\ref{diag_aux_algebric}) just replacing $\alpha$ and $n_1$ by $\alpha_1$ and $n_1'$, respectively. In analogous way of the first part of the proof, let $\gamma$ a path in $F_2(N)_B$ whose image is contained in $F_2(S_N)$ such that $\gamma(0) = (n_1,n_2)$ and $\gamma(1) = (n'_1,n'_2)$. This path induces an isomorphism $\gamma_\#\colon \pi_1(F_2(N)_B,(n'_1,n'_2)) \to \pi_1(F_2(N),(n_1,n_2))$. Also we have the isomorphisms ${\overline{\gamma}}_\#\colon \pi_1(D_2(N)_B,\overline{(n'_1,n_2')}) \to \pi_1(D_2(N)_B,\overline{(n_1,n_2)})$ and $(\gamma_1)_\#\colon \pi_1(N,n_1') \to \pi_1(N,n_1)$ obtained from the paths $\overline{\gamma} = \pi_{\tau'} \circ\gamma$ and $\gamma_1 = p_1 \circ \gamma$, respectively. Consider the homomorphisms $\varphi_2\colon \pi_1(M,m_1) \to F_2(N,(n_1,n_2))$ and $\psi_2\colon \pi_1(M_\tau,\overline{m_1}) \to D_2(N,\overline{(n_1,n_2)})$ defined by $\varphi_2 = \gamma_\# \circ \varphi_1$ and $\psi_2 = \overline{\gamma}_\# \circ \psi_1$, respectively. By Theorem~\ref{th:set_homotopy}, there exists a pointed homotopy class $\alpha_2 \in [M,m_1;N,n_1]_B$ such that $(\alpha_2)_\# = (p_1)_\# \circ \varphi_2$. By construction, we obtained the diagram~(\ref{diag_borsuk_braid}) just replacing $\alpha$, $\varphi$ and $\psi$ by $\alpha_2$, $\varphi_1$ and $\psi_1$, respectively. Note that $(\alpha_2)_\# = (\gamma_1)_\# \circ (\alpha_1)_\# = (\gamma_1)_\# \circ f_\#$ which implies that $(\alpha_2)_\free = \beta$ by Theorems~\ref{th:set_homotopy} and~\ref{th:criterion_homotopy}. By the hypothesis we have that $(\alpha_2)_\free = \alpha_\free$. Again by Theorem~\ref{th:set_homotopy}, there exists an element $\omega \in \iota_\#(\pi_1(S_N,n_1))$ such that the homomorphisms $\alpha_\#,(\alpha_2)_\#\colon \pi_1(M,m_1) \to \pi_1(N,n_1)$ satisfies and $\alpha_\# (v) = \omega (\alpha_2)_\#(v) \omega^{-1}$ for all $v \in \pi_1(M,m_1)$. By \cite[Theorem~1.4]{Bir}, the homomorphism $(p_1)_\#\colon P_2(S_N) \to \pi_1(S_N,n_1)$ is surjective and therefore, there exists a braid $b \in P_2(S_N)$ such that $(p_1)_\#(b)=\omega$. Taking into account the diagram~(\ref{diag-fiber-bundle-induced-2}), we may suppose without loss of generality that $b$ belongs to $P_2(N)_B$ and $B_2(N)_B$. Finally, we define $\varphi\colon \pi_1(M,m_1) \to P_2(N)_B$ and $\psi\colon \pi_1(M_\tau,\overline{m_1}) \to B_2(N)_B$ by $\varphi(z_1) = b \varphi_2(z_1) b^{-1}$ and $\psi(z_2) = b \psi_2(z_2) b^{-1}$, respectively, and we leave to the reader to check that we obtain the diagram~(\ref{diag_borsuk_braid}). Thus $\alpha$ does not have the Borsuk-Ulam property.
\end{proof}

\section{Applications for some torus bundle over $\mathbb{S}^1$}

In this section we will use the previous theory to classify    homotopy classes over  $\mathbb{S}^1$ for  two $2$-torus bundles over $\mathbb{S}^1$ 
with respect to  the Borsuk-Ulam property. 

Let $\mathbb{T}^2$,  the 2-torus, defined as the quotient space $ {\mathbb{R} \times \mathbb{R}}/{\mathbb{Z} \times \mathbb{Z}} $. We denote by $(x, y)$ the elements of  $\mathbb{R} \times \mathbb{R}$ and by $[(x,y)]$ the elements in $\mathbb{T}^2$.
Let $$MA = \frac{\mathbb{T}^2 \times [0,1]}{([(x,y)],0) \sim (\left[A \left(^{x}_{y} \right) \right],1)}$$ be the quotient space, where $A$ is a homeomorphism of $\mathbb{T}^2$ induced by an operator in $\mathbb{R}^{2}$ that preserves $ \mathbb{Z} \times \mathbb{Z}$. The space $MA$ is a fiber bundle over the circle $\mathbb{S}^{1}$ where the fiber is the torus. $A$ is called the monodromy matrix of $MA.$ The projection map $p: MA \longrightarrow \mathbb{S}^{1}$ is given by $p(<(x,y),t>) = <t> \in I/0 \sim 1 = \mathbb{S}^{1}.$ For more details see \cite[Page 7]{GonPenVie04}.
Let $s_{0}:\mathbb{S}^{1} \longrightarrow MA$ be given by $s_{0}(<t>) = <(0, 0),t>.$ This is a section of the fiber bundle $p: MA \longrightarrow \mathbb{S}^{1}.$ 
From \cite[Proposition 1.7]{GonPenVie04} we have;

\begin{proposition}  We have a short exact sequence $ 1 \longrightarrow \pi_{1}(\mathbb{T}^2) \longrightarrow \pi_{1}(MA) \longrightarrow \pi_{1}(\mathbb{S}^{1}) \longrightarrow 1 $ which splits, and the action $ \mathbb{Z} \longrightarrow  Aut(\pi_{1}(\mathbb{T}^2))$ which comes from the section $s_{0}$ is given by  $c.\gamma = <s_0> \gamma <s_0>^{-1} = A_{\#}(\gamma)$, where $c = p_{\#}(<s_{0}>)$ is the generator of $\pi_{1}(\mathbb{S}^{1}). $ Hence $\pi_{1}(MA) \cong \pi_{1}(\mathbb{T}^2) \rtimes   \mathbb{Z} $
	a semidirect product.
\end{proposition}

\smallskip

Consider the loops in $MA$  given by; $\alpha(t) = <[(t,0)],0> $, $\beta(t) = $ $ <[(0,t)],0> $ and $c(t) =$ $ <[(0,0)],t> $ for $t \in [0,1]$. We denote by $D$ the matrix of the homomorphism induced on the fundamental group by the restriction of $f$ to the fiber $\mathbb{T}^2.$ From 
\cite[Theorem 2.1]{GonPenVie04} we have the following theorem that provides a relationship between the matrices $A$ and $D$, where 
$$ A = { \left[ \begin{array}{cc} a_{11} & a_{12} \\ a_{21} & a_{22} \\ \end{array} \right]} \,\,\,\, and \,\,\,\, 
D = { \left[ \begin{array}{cc} b_{1} & b_{3} \\ b_{2} & b_{4} \\ \end{array} \right]}. $$

\begin{proposition} \label{GonPenVie04-theorem1}
	
	\begin{enumerate}[(1)]
		
		\item $\pi_{1}(MA,0) = \langle \alpha, \beta, c | [\alpha,\beta] = 1, c \alpha c^{-1} = \alpha^{a_{11}} \beta^{a_{21}} , 
		c \beta c^{-1}  $ $ = \alpha^{a_{12}} \beta^{a_{22}} \rangle $. 
		
		\item $D$  commutes with $A$.

		\item The induced homomorphism $f_{\#}: \pi_{1}(MA,0) \longrightarrow \pi_{1}(MA,0) $ is given by; $f_{\#}(\alpha) = \alpha^{b_{1}} \beta^{b_{2}} $, $f_{\#}(\beta) = \alpha^{b_{3}} \beta^{b_{4}} $, $f_{\#}(c) = \alpha^{c_{1}} \beta^{c_{2}} c .$
		
	\end{enumerate}
	
\end{proposition}


Now we consider the case $M_A  \to \mathbb{S}^1$ for  $A = \left[\begin{array}{cc}
		1 & 1 \\
		0 & 1 \\
	\end{array}\right].$
\begin{theorem} \label{main-result-1} 
Let $MA \to \mathbb{S}^1$ be the torus bundle over $\mathbb{S}^{1}$ where $A = \left[\begin{array}{cc}
		1 & 1 \\
		0 & 1 \\
	\end{array}\right].$ $MA$ has just one free involution $\tau$ which is over $\mathbb{S}^{1}$ with $MA_{\tau} \simeq MB,$ where 
	$B = \left[\begin{array}{cc}
		1 & 2 \\
		0 & 1 \\
	\end{array}\right].$ 
Given a homotopy class $\delta = [f] \in [MA,0; MA,0]_{\mathbb{S}^{1}}$  the homomorphism $f_{\#}$ is given by $$f_{\#}(\alpha) = \alpha^{r}, \,\, f_{\#}(\beta) = \alpha^{s} \beta^{r} \,\,\, and \,\,\, f_{\#}(c) = \alpha^{u} \beta^{v}c.$$
	
	$\delta$ has the Borsuk--Ulam property with respect to $\tau$ if and only if $r$ is even.
\end{theorem}
\begin{proof}
	The existence of involution $\tau$ follows from \cite[Theorem V]{Sak}. The fact that $\tau$ is over $\mathbb{S}^{1}$ follows from \cite[Pages 167 and  182]{Sak}. The expression of $f_{\#}$ follows from Proposition \ref{GonPenVie04-theorem1}. Now we construct homomorphisms $\psi$ and $\varphi$ in the Diagram \ref{diag_borsuk_braid} when $r$ is odd and prove that such homomorphisms can not exist when $r$ is even.

We observe that  the element $\rho_{1,1},$ given in Section \ref{appendix}, when included in $\mathbb{T}^{2} = {\mathbb{R} \times \mathbb{R}}/{\mathbb{Z} \times \mathbb{Z}}$ is homotopic to $[(-t,0)].$ Therefore, the element $\alpha$ given in Proposition \ref{GonPenVie04-theorem1} is homotopic to $\tilde{\alpha}^{-1},$ where 
$\tilde{\alpha} = <\rho_{1,1}, 0>$ was taken as a generator of $\pi_{1}(MA, <x_1,0>)$ in Section \ref{appendix}. For facilitate the computations we will make a change on generators in presentation of $\pi_{1}(MA,<x_{1},0>)$ given in Proposition \ref{GonPenVie04-theorem1}.

Let $\tilde{\alpha} = \alpha^{-1}$.      From Theorem \ref{th:set_homotopy} there exist a homeomorphism $h: (MA, <x_1,0>) \to (MA, <x_1,0>) ,$ over $S^{1},$ such 
that $h(\alpha) = \tilde{\alpha},$ \,  $h(\beta) = \beta$ and $h(c) = c.$ From matrix $A$ and Proposition \ref{GonPenVie04-theorem1} we have;
$$ c \tilde{\alpha} c^{-1} = c \alpha^{-1} c^{-1} = \alpha^{-1} = \tilde{\alpha}, \,\,\,\,and \,\,\,\,\, c \beta c^{-1} = \alpha \beta = (\alpha^{-1})^{-1} \beta = \tilde{\alpha}^{-1} \beta. $$
Therefore 	
$$ \pi_{1}(MA,<x_{1},0>) = <\tilde{\alpha}, \beta,c| \tilde{\alpha} \beta \tilde{\alpha}^{-1} \beta^{-1} = 1, c \tilde{\alpha} c^{-1} = \tilde{\alpha}, c \beta c^{-1} = \tilde{\alpha}^{-1}\beta>. $$
In this new presentation we have; 
$$f_{\#}(\tilde{\alpha}) = \tilde{\alpha}^{r}, \,\, f_{\#}(\beta) = \tilde{\alpha}^{-s} \beta^{r} \,\,\, and \,\,\, f_{\#}(c) = \tilde{\alpha}^{-u} \beta^{v}c.$$
	
If $\tilde{\tau} = h \circ \tau \circ h^{-1}$ then by the same argument above we have;	
$$\pi_{1}(MA_{\tilde{\tau}}, \overline{<x_{1},0>}) = <a, b , \hat{c}| a b a^{-1} b^{-1} = 1, \hat{c} a \hat{c}^{-1} = a, 
	\hat{c} b \hat{c}^{-1} = a^{-2} b>. $$
	The presentations of $P_{2}(MA)_{\mathbb{S}^{1}}$ and $B_{2}(MA)_{\mathbb{S}^{1}}$ are given in Section \ref{appendix}. 
	We can consider the epimorphism  $\theta_{\tilde{\tau}}: \pi_{1}(MA_{\tilde{\tau}}, \overline{<x_{1},0>}) \rightarrow \ztwo $ given by; 
	$$ \theta_{\tilde{\tau}}: \left \{\begin{array}{l}
		a \mapsto \bar{1} \\
		b \mapsto \bar{0} \\
		\hat{c} \mapsto \bar{0} \\
	\end{array} \right.
	$$
	
	The homomorphism $\pi_{\tilde{\tau}} : \pi_{1}(MA, <x_{1},0>) \rightarrow \pi_{1}(MA_{\tilde{\tau}}, \overline{<x_{1},0>})$ is given by
	$$ \pi_{\tilde{\tau}}: \left \{\begin{array}{l}
		\tilde{\alpha} \mapsto a^{2} \\
		\beta \mapsto b \\
		c \mapsto \hat{c} \\
	\end{array} \right.
	$$
	
	From Lemma \ref{equiv_maps}, if there exists a homomorphism $\psi$ in the Diagram \ref{diag_borsuk_braid} then we must have
	$$\left\{\begin{array}{lll}
		\psi(a) = W_{1}(\rho_{ij}) & if & \theta_{\tilde{\tau}}(a) = \bar{0} \\
		\psi(a) = W_{1}(\rho_{ij})\sigma & if & \theta_{\tilde{\tau}}(a) = \bar{1} \\
	\end{array} \right., \,\,\,\,
	\left\{\begin{array}{lll}
		\psi(b) = W_{2}(\rho_{ij}) & if & \theta_{\tilde{\tau}}(b) = \bar{0} \\
		\psi(b) = W_{2}(\rho_{ij})\sigma & if & \theta_{\tilde{\tau}}(b) = \bar{1} \\
	\end{array} \right.\,\,\,\,\,  and $$
	$$
	\left\{\begin{array}{lll}
		\psi(\hat{c}) = W_{3}(\rho_{ij})\bar{c} & if & \theta_{\tilde{\tau}}(\hat{c}) = \bar{0} \\
		\psi(\hat{c}) = W_{3}(\rho_{ij})\bar{c} \sigma & if & \theta_{\tilde{\tau}}(\hat{c}) = \bar{1} \\
	\end{array} \right.
	$$  
	where $W_{j}(\rho_{ij})$ are words in $P_{2}(\mathbb{T}^{2}).$
	By presentations of $\pi_{1}(MA_{\tau}, \overline{<x_{1},0>}),$ $P_{2}(MA)_{\mathbb{S}^{1}}$ and $B_{2}(MA)_{\mathbb{S}^{1}},$ there exists a homomorphism 
	$\psi$ in the Diagram \ref{diag_borsuk_braid}, which makes it commutative, if and only if there are elements $Z_{1}, Z_{2},Z_{3} \in B_{2}(MA)_{S^{1}} $ such that 
	\begin{equation} \left \{\begin{array}{l}
			\psi(a) = Z_{1}, \\
			\psi(b) = Z_{2}, \\
			\psi(\hat{c}) = Z_{3},  \\
			\psi(aba^{-1}b^{-1}) = 1 \,\,\,\, \Rightarrow \,\,\,\, Z_{1}Z_{2}Z_{1}^{-1}Z_{2}^{-1} = 1, \\
			\psi(\hat{c}a \hat{c}^{-1}) = \psi(a) \,\,\,\, \Rightarrow \,\,\,\, Z_{3} Z_{1} Z_{3}^{-1} = Z_{1}, \\
			\psi(\hat{c} b \hat{c}^{-1}) = \psi(a^{-2}b) \,\,\,\, \Rightarrow \,\,\,\, Z_{3} Z_{2} Z_{3}^{-1} = Z_{1}^{-2}Z_{2}. \\ 
		\end{array} \right.
	\end{equation}
	where $Z_{1} = P_{1}\sigma,$ $Z_{2} = P_{2},$  $Z_{3} = P_{3}\bar{c},$ and $P_{1},P_{2},P_{3} \in P_{2}(\mathbb{T}^{2}).$ That is, there exists a homomorphism $\psi$ in the Diagram \ref{diag_borsuk_braid}, which makes it commutative, if and only if there are elements 
	$P_{1},P_{2},P_{3} \in P_{2}(\mathbb{T}^{2})$ such that
	\begin{equation} \label{systemI-1} \left \{\begin{array}{l}
			P_{1} \sigma P_{2} \sigma^{-1} = P_{2}P_{1}, \\
			P_{3} \bar{c} P_{1} \bar{c}^{-1} = P_{1} \sigma P_{3} \sigma^{-1} , \\
			P_{3} \bar{c} P_{2} \bar{c}^{-1} = (P_{1} \sigma P_{1} \sigma^{-1} B)^{-2}P_{2}P_{3}, \\ 
		\end{array} \right.
	\end{equation}
	where $\sigma^{2} = B.$ Therefore, our problem is equivalent to solve a system in the pure braid group of the torus. Suppose that there exists a solution in the System 
	\eqref{systemI-1}. This guarantee the existence of homomorphisms $\psi$ and $\varphi$ in the Diagram \ref{diag_borsuk_braid}. 
	With this hypothesis we must have
	$$ \left \{\begin{array}{l}
		(\pi_{\tau^{'}})_{\#} \circ \varphi(\tilde{\alpha}) = \psi \circ ({\pi_{\tilde{\tau}}})_{\#}(\tilde{\alpha}) = \psi(a)^{2} = (P_{1}\sigma)^{2} = 
		(P_{1} \sigma P_{1} \sigma^{-1} B), \\
		(\pi_{\tau^{'}})_{\#} \circ \varphi(\beta) = \psi \circ ({\pi_{\tilde{\tau}}})_{\#}(\beta) = \psi(b) = P_{2}, \\
		(\pi_{\tau^{'}})_{\#} \circ \varphi(c) = \psi \circ ({\pi_{\tilde{\tau}}})_{\#}(c) = \psi(\hat{c}) = P_{3}\bar{c} ,\\
	\end{array} \right. \hspace{0.3cm} and \hspace{0.3cm}
	\left \{\begin{array}{l}
		\rho_{\#} \circ \varphi(\tilde{\alpha}) = f_{\#}(\tilde{\alpha}),  \\   
		\rho_{\#} \circ \varphi(\beta) = f_{\#}(\beta),  \\   
		\rho_{\#} \circ \varphi(c) = f_{\#}(c).  \\ 
	\end{array} \right.
	$$
	Since $(\pi_{\tau^{'}})_{\#}: P_{2}(MA)_{S^{1}} \rightarrow B_{2}(MA)_{S^{1}}$ is an inclusion, we obtain
	\begin{equation} \label{equationI-1} 
		\left \{\begin{array}{l}
			(p_1)_{\#}( P_{1} \sigma P_{1} \sigma^{-1}) = \tilde{\alpha}^{r}, \\
			(p_1)_{\#}( P_{2} ) = \tilde{\alpha}^{-s} \beta^{r}, \\
			(p_1)_{\#}( P_{3}) = \tilde{\alpha}^{-u} \beta^{v}. \\ 
		\end{array} \right.
	\end{equation}
	
	From Lemma \ref{lemma-gamma}  we have;
	\begin{equation} \label{equationI-2} \left \{\begin{array}{l}
			P_{1} = (x^{a_{1}}y^{b_{1}}A_{1}; m_{1}, n_{1}), \\
			P_{2} = (x^{a_{2}}y^{b_{2}}A_{2}; m_{2}, n_{2}), \\ 
			P_{3} = (x^{a_{3}}y^{b_{3}}A_{3}; m_{3}, n_{3}), \\
		\end{array} \right.
	\end{equation}
	where $A_{1} = \displaystyle \prod_{i=1} (x^{e_{i}} y^{f_{i}} B^{r_{i}} 
	y^{-f_{i}} x^{-e_{i}}),$ $A_{2} = \displaystyle \prod_{i=1} (x^{h_{i}} y^{j_{i}} B^{s_{i}} 
	y^{-j_{i}} x^{-h_{i}})$ and 
	$A_{3} = \displaystyle \prod_{i=1} (x^{k_{i}} y^{l_{i}} B^{t_{i}} 
	y^{-l_{i}} x^{-k_{i}}).$
	
	By definition of $(p_1)_{\#}$ we have $(p_1)_{\#}(P_{i}) = \tilde{\alpha}^{m_{i}} \beta^{n_{i}}, i=1,2,3.$ 
	Projecting the first equation of the System \eqref{equationI-1} in the subgroup $\mathbb{Z} \oplus \mathbb{Z}$ and using the expressions of 
	$P_{i}$ in \eqref{equationI-2} and Proposition \ref{prop-conjugations} we must have $(m_{1},n_{1}) + (a_{1}, b_{1}) + (m_{1},n_{1}) = (r,0)$ 
	which implies $a_{1} = -2m_{1}+r$ and $b_{1}=-2n_{1}.$ Doing the same thing with the second and third equations of \eqref{equationI-1} we 
	obtain $m_{2}=-s,$ $n_{2}=r$ and $m_{3} = -u,$  $n_{3} = v,$ respectively.

	Projecting the first equation of System \eqref{systemI-1} in the subgroup $\mathbb{Z} \oplus \mathbb{Z}$ of $P_{2}(\mathbb{T}^{2})$ we obtain; 
	$ (m_{1},n_{1}) + (a_{2},b_{2}) + (s,r) = (s,r) + (m_{1},n_{1})$ which implies $a_{2} = b_{2} = 0.$ 
	Projecting the second equation of System \eqref{systemI-1} in  $\mathbb{Z} \oplus \mathbb{Z} $ we obtain; $ (-u,v) + (m_{1},0) + (-n,1)^{n_{1}} = (m_{1},n_{1}) + (a_{3}, b_{3}) + (-u,v)$ which implies $-n_{1}=a_{3}$ and $b_{3}=0.$ Thus, for some $m_{1},n_{1} \in \mathbb{Z},$ 
	we can write
	\begin{equation} \label{equationI-3} \left \{\begin{array}{l}
			P_{1} = (x^{-2m_{1}+r}y^{-2n_{1}}A_{1}; m_{1}, n_{1}), \\
			P_{2} = ( A_{2}; -s, r), \\ 
			P_{3} = (x^{-n_{1}} A_{3}; -u, v). \\
		\end{array} \right.
	\end{equation}
	
	Suppose $r = 2k-1.$ In this situation we states that the following is a solution to the System \eqref{systemI-1}.  
	$$
	\left \{\begin{array}{l}
		P_{1} = (x^{r} [x^{-k}(xB^{-1})^{k}]; 0, 0), \\
		P_{2} = ( \id ; -s, r), \\ 
		P_{3} = (\id; -u, v). \\
	\end{array} \right.
	$$
	
	In fact, by construction of $P_{1},P_{2},P_{3}$ in \eqref{equationI-3}, to verify the validity of System \eqref{systemI-1} is enough 
	to verify the validity the system projected in the free group $F(x,y).$ In this situation that $A_{2} = A_{3} = \id$ and 
	$A_{1} = x^{-k}(xB^{-1})^{k},$ then it suffices to check, in the third equation, that $P_{1} \sigma P_{1} \sigma^{-1} B = \id$ in $F(x,y),$ that is,
	$x^{r}A_{1}\sigma x^{r}A_{1} \sigma^{-1} B = \id.$ From Proposition \ref{prop-conjugations} we have; 
	$x^{r}A_{1}\sigma x^{r}A_{1} \sigma^{-1} B=$  $x^{r}x^{-k}(xB^{-1})^{k}\sigma x^{r}x^{-k}(xB^{-1})^{k} \sigma^{-1} B=$ 
	$x^{k-1}(xB^{-1})^{k}\sigma x^{k-1}(xB^{-1})^{k} \sigma^{-1} B=$
	$x^{k-1}(xB^{-1})^{k}(\sigma x^{k-1} \sigma^{-1})(\sigma x  $ $\sigma^{-1}\sigma B^{-1}\sigma^{-1})^{k}  B=$
	$x^{k-1}(xB^{-1})^{k} (Bx^{-1})^{k-1} (Bx^{-1}  B^{-1})^{k}  B=$
	$x^{k-1}(xB^{-1})^{k} $ $ (xB^{-1})^{1-k} Bx^{-k}B^{-1}  B=$
	$x^{k-1} $ $ x B^{-1}  Bx^{-k}B^{-1}  B=$
	$x^{k-1}  x x^{-k} = \id.$

	Now we suppose $r = 2k.$ In this situation, we will prove that the third equation of the System \eqref{systemI-1}, projected in the 
	subgroup $F(x,y),$ has not solution. Projecting the third equation of the System \eqref{systemI-1} in $F(x,y)$ and use the 
	Proposition \ref{prop-conjugations} we obtain;
	\begin{equation} \label{equationI-4} \begin{array}{l}
			x^{-n_{1}} (A_{3} (cA_{2}c^{-1}) A_{3}^{-1})x^{n_{1}} A_{2}^{-1} =  
			[x^{-2m_{1}+r}y^{-2n_{1}} A_{1}(x,y) y^{2n_{1}}x^{2m_{1}-r} (x^{-2m_{1}+r}y^{-2n_{1}}xy^{-1} x^{2m_{1}-r} \\ 
			y^{2n_{1}} yx^{-1}) xy^{-1} A_{1}(x^{-1}y^{-1}) yx^{-1}B]^{-1}. \\
		\end{array} 
	\end{equation}
	
	Note that $(x^{-2m_{1}+r}y^{-2n_{1}}xy^{-1} x^{2m_{1}-r}y^{2n_{1}} yx^{-1}) = (x^{-2m_{1}+2k}[y^{2n_{1}},x] x^{2m_{1}-2k}) 
	(x^{-2m_{1}+2k+1}$ \\
	\noindent $ [y^{2n_{1}-1}, x^{2m_{1}-2k}] x^{2m_{1}-2k-1}) .$
	From Lemma \ref{lemma-epsilon} we obtain $\mathcal{E}(x^{-2m_{1}+r}y^{-2n_{1}}xy^{-1} x^{2m_{1}-r}y^{2n_{1}} yx^{-1}) = $ $
	2[-n_{1}$ $+(-m_{1}+k)(2n_{1}-1)].$
	Now applying the homomorphism $\mathcal{E}$ in the Equation \eqref{equationI-4} we obtain;
	$$
	0 = 2(\displaystyle\sum_{i}r_{i}) + 2[-n_{1}+(-m_{1}+k)(2n_{1}-1)] + 1. 
	$$ 
	But the last equation is a contradiction. Thus, when $r$ is even the System \eqref{systemI-1} can not have solution.
Now, using a version of \cite[Proposition 8]{GonGuaLaa} for fiber preserving maps over $S^{1},$ the result follows. 
\end{proof}


Next we present the  result for the trivial bundle $\mathbb{T}^{2} \times \mathbb{S}^{1}\to \mathbb{S}^{1}.$  

Consider the unitary circle $\mathbb{S}^1 = \{ z \in \mathbb{C} \, | \,  |z|=1 \} = \{ e^{\theta i} \, ; \, \theta \in \mathbb{R} \}$, the $2$-torus $\mathbb{T}^2 = \mathbb{S}^1 \times \mathbb{S}^1$ and the $3$-torus $\mathbb{T}^3 = \mathbb{S}^1 \times \mathbb{S}^1 \times \mathbb{S}^1$. Recall that 
 the projection map $p\colon \mathbb{T}^3 \to \mathbb{S}^1$  given by $p(x,y,z) = z$ is the trivial 
 2-torus bundle over $\mathbb{S}^1$
 with   fiber  $\mathbb{T}^2$, where the natural inclusion $i\colon \mathbb{T}^2 \to \mathbb{T}^3$ is given by $i(x,y)=(x,y,1)$. We fix $m_1 = 1$, $m_2 = (1,1)$ and $m_3 = (1,1,1)$ as the base points of $\mathbb{S}^1$, $\mathbb{T}^2$ and $\mathbb{T}^3$, respectively, the homotopy class of the loop $c = e^{2\pi i t}$, $t \in [0,1]$, as the canonical generator of $\pi_1(\mathbb{S}^1)$, and the presentations $\pi_1(\mathbb{S}^1) = \l c \, | \, \, \, \r$, $\pi_1(\mathbb{T}^2) = \l \alpha,\beta \, | \, [\alpha,\beta]=1 \r$ and $\pi_1(\mathbb{T}^3) = \l \alpha,\beta , c\, | \, [\alpha,\beta]=[\alpha,c]=[\beta,c]=1 \r$.

Note that if $f\colon (\mathbb{T}^3,m_3) \to (\mathbb{T}^3,m_3)$ is a map over $\mathbb{S}^1$, then $f_\#\colon \pi_1(\mathbb{T}^3) \to \pi_1(\mathbb{T}^3)$ has the form
\begin{equation}\label{diag:f_induced}
	f_\#(\alpha) = \alpha^{r_1} \beta^{r_2} \qquad f_\#(\beta) = \alpha^{r_3} \beta^{r_4} \qquad f_\#(c) = \alpha^{u} \beta^{v}c,
\end{equation}
for unique $r_1,r_2,r_3,r_4,u,v \in \mathbb{Z}$. Also, using Theorem~\ref{th:set_homotopy} and the fact that $\pi_1(\mathbb{T}^2)$ is an Abelian group  we have the following sequence of bijections:

\begin{equation}\label{diag:free_pointed}\xymatrix{
		[\mathbb{T}^3,\mathbb{T}^3]_{\mathbb{S}^1} \ar[r]^-{\Lambda^{-1}} \ar@/_0.8cm/[rr]^-{\Delta}  & [\mathbb{T}^3,m_3;\mathbb{T}^3,m_3]_{\mathbb{S}^1} \ar[r]^-{\Gamma} & \hom (\pi_1(\mathbb{T}^3),\pi_1(\mathbb{T}^3))_{\mathbb{S}^1} \ar[r]^-{\Omega} & M_{2 \times 6}(\mathbb{Z}).
}\end{equation}

Given a homotopy class $\delta \in [\mathbb{T}^3,\mathbb{T}^3]_{\mathbb{S}^1}$, we define $\delta_\# = (\Omega \circ \Delta)(\delta)
= \left[ \begin{array}{ccc} 
	r_1 & r_3 & u \\
	r_2 & r_4 & v
\end{array}\right]$.

\begin{theorem}\label{th:tau}
	Let $\tau_1,\tau_2\colon \mathbb{T}^3 \to \mathbb{T}^3$ be the free involutions over $\mathbb{S}^1$ given by
	$$\tau_1(e^{\theta_1i},e^{\theta_2i},e^{\theta_3i})=(e^{(\theta_1+\pi)i},e^{\theta_2i},e^{\theta_3i})$$
	$$\tau_2(e^{\theta_1i},e^{\theta_2i},e^{\theta_3i})=(e^{(\theta_1+\pi)i},e^{(\pi-\theta_2)i},e^{\theta_3i}),$$
	and let $\delta \in [\mathbb{T}^3, \mathbb{T}^3]_{\mathbb{S}^1}$ be a homotopy class. Then
	\begin{enumerate}
		\item $\delta$ does not have the Borsuk-Ulam property with respect to $\tau_1$.
		
		\item $\delta$ has the Borsuk--Ulam property with respect to $\tau_{2}$ if and only if $(r_1, r_2) \neq (0,0),$ and $r_3$ and $r_4$ are both even. 
	\end{enumerate}
\end{theorem}

Before to prove the Theorem~\ref{th:tau}, we will establish some presentations and a Lemma.

\begin{remark}\label{rem:orbit}
	Note that the orbit space $\mathbb{T}^3_{\tau_j} = \begin{cases}
		\mathbb{T}^3, \text{ if } j = 1, \\
		\mathbb{K}^2 \times \mathbb{S}^1, \text{ if } j = 2,
	\end{cases}$
	where $\mathbb{K}^2$ denotes de Klein bottle. The diagram~(\ref{seq_tau}) applied to the involutions $\tau_1$ and $\tau_2$ give us the following diagrams:
	
	\begin{equation*}\xymatrix{
			\pi_1 ( \mathbb{T}^3) = \l \alpha,\beta,c \, | \, [\alpha,\beta]=[\alpha,c]=[\beta,c] = 1 \r \ar[d]^-{(\pi_{\tau_1})_\#} \ar[r]^-{p_\#} &  \pi_1(\mathbb{S}^1) = \l c \, | \, \, \, \r \\
			\pi_1 ( \mathbb{T}^3 ) = \l \alpha,\beta,c \, | \, [\alpha,\beta]=[\alpha,c]=[\beta,c] = 1 \r \ar[r]^-{\theta_{\tau_1}} \ar[ru]_-{\overline{p_{\tau_1}}_\#} & \ztwo \\
	}\end{equation*}
	
	\begin{multicols}{4}
		$p_\#\colon \begin{cases}
			\alpha \mapsto 1 \\
			\beta \mapsto 1 \\
			c \mapsto c
		\end{cases}$
		
		$(\pi_{\tau_1})_\#\colon \begin{cases}
			\alpha \mapsto \alpha^2 \\
			\beta \mapsto \beta \\
			c \mapsto c
		\end{cases}$
		
		$\overline{p_{\tau_1}}_\#\colon \begin{cases}
			\alpha \mapsto 1 \\
			\beta \mapsto 1 \\
			c \mapsto c
		\end{cases}$
		
		$(\theta_{\tau_1})\colon \begin{cases}
			\alpha \mapsto \overline{1} \\
			\beta \mapsto \overline{0} \\
			c \mapsto \overline{0}
		\end{cases}$
	\end{multicols}
	
	\begin{equation*}\xymatrix{
			\pi_1 ( \mathbb{T}^3) = \l \alpha,\beta,c \, | \, [\alpha,\beta]=[\alpha,c]=[\beta,c] = 1 \r \ar[d]^-{(\pi_{\tau_2})_\#} \ar[r]^-{p_\#} &  \pi_1(\mathbb{S}^1) = \l c \, | \, \, \, \r \\
			\pi_1 ( \mathbb{K}^2 \times \mathbb{S}^1) = \l a,b,c \, | \, abab^{-1}=1, [a,c]=[b,c] = 1 \r \ar[r]^-{\theta_{\tau_2}} \ar[ru]_-{\overline{p_{\tau_2}}_\#} & \ztwo \\
	}\end{equation*}
	
	\begin{multicols}{4}
		$p_\#\colon \begin{cases}
			\alpha \mapsto 1 \\
			\beta \mapsto 1 \\
			c \mapsto c
		\end{cases}$
		
		$(\pi_{\tau_2})_\#\colon \begin{cases}
			\alpha \mapsto b^2 \\
			\beta \mapsto a \\
			c \mapsto c
		\end{cases}$
		
		$\overline{p_{\tau_2}}_\#\colon \begin{cases}
			a \mapsto 1 \\
			b \mapsto 1 \\
			c \mapsto c
		\end{cases}$
		
		$(\theta_{\tau_2})\colon \begin{cases}
			a \mapsto \overline{0} \\
			b \mapsto \overline{1} \\
			c \mapsto \overline{0}
		\end{cases}$
	\end{multicols}
\end{remark}

\begin{remark}\label{rem:presentation_braids}
	In the diagram~(\ref{diag-fiber-bundle-geo}), take $S_N = \mathbb{T}^2$, $N = \mathbb{T}^3$, $B = \mathbb{S}^1$ and $q = p$. Note that the map $g\colon \mathbb{S}^1 \to F_2(\mathbb{T}^3)_{\mathbb{S}^1}$ defined by $g(e^{\theta_i}) = ((e^{\theta i},e^{\theta i},e^{\theta i}),(e^{(\theta +\pi) i},e^{\theta i},e^{\theta i}))$ is a section of $p \circ p_1$. Also the map $\overline{g}\colon \mathbb{S}^1 \to D_2(\mathbb{T}^3)_{\mathbb{S}^1}$ given by $g(e^{\theta_i}) = \{ (e^{\theta i},e^{\theta i},e^{\theta i}),(e^{(\theta +\pi) i},e^{\theta i},e^{\theta i})\}$ is a section of $\overline{p \circ p_1}$. Then, from the diagram~(\ref{diag-fiber-bundle-induced-2}) we obtain the following commutative diagram, where the horizontal lines are split exact sequences.
	\begin{equation*}\xymatrix{ 
			1 \ar[r] & P_2(\mathbb{T}^2) \ar@{^{(}->}[r] \ar@{^{(}->}[d] & P_2(\mathbb{T}^3)_{\mathbb{S}^1} \ar[r]^-{(p\circ p_1)_\#} \ar@{^{(}->}[d]  & \pi_1(\mathbb{S}^1) \ar@/_0.8cm/[l]_-{g_\#} \ar[r] \ar[d]^-{\rm Id} & 1 \\
			1\ar[r] & B_2(\mathbb{T}^2) \ar@{^{(}->}[r] &  B_2(\mathbb{T}^3)_{\mathbb{S}^1} \ar[r]^-{\overline{p \circ p_1}_\#} & \pi_1(\mathbb{S}^1) \ar@/^0.8cm/[l]_-{\overline{g}_\#} \ar[r] & 1 .
	}\end{equation*}
	We leave to the reader to check that the actions by conjugation induced by $g_\#$ and $\overline{g}_\#$ on $P_2(\mathbb{T}^2)$ and $B_2(\mathbb{T}^2)$, respectively, are trivial. Then $P_2(\mathbb{T}^3)_{\mathbb{S}^1} \cong P_2(\mathbb{T}^2) \oplus \pi_1(\mathbb{S}^1)$ and $B_2(\mathbb{T}^3)_{\mathbb{S}^1} \cong B_2(\mathbb{T}^2) \oplus \pi_1(\mathbb{S}^1)$. Moreover, up to isomorphism, the induced homomorphism of $p_1\colon F_2(\mathbb{T}^3)_{\mathbb{S}^1}  \to \mathbb{T}^3$ on the fundamental groups is the direct sum of the induced homomorphisms of $p_1\colon F_2(\mathbb{T}^2) \to \mathbb{T}^2$ and ${\rm Id}\colon \mathbb{S}^1 \to \mathbb{S}^1$.
\end{remark}

\begin{lemma}\label{lem:t3_t2}
	For each $j \in \{1,2\}$, let $\bar{\tau_j}\colon \mathbb{T}^2 \to \mathbb{T}^2$ be the free involution defined by $\bar{\tau_j}(e^{\theta_1 i}, e^{\theta_2 i}) = \tau_j (e^{\theta_1 i}, e^{\theta_2 i},1)$. Given a pointed homotopy class $\delta = [f] \in [\mathbb{T}^3,m_3;\mathbb{T}^3,m_3]_{\mathbb{S}^1}$, let $\bar{\delta} \in [\mathbb{T}^2,m_2;\mathbb{T}^2,m_2]$ be the pointed homotopy class whose $f|_{\mathbb{T}^2}\colon {\mathbb{T}^2} \to {\mathbb{T}^2}$ is a representative map. Then $\left(f|_{\mathbb{T}^2}\right)_\# \colon \pi_1(\mathbb{T}^2) \to \pi_1(\mathbb{T}^2)$ satisfies $\left(f|_{\mathbb{T}^2}\right)_\#(\alpha) = \alpha^{r_1} \beta^{r_2}$ and  $\left(f|_{\mathbb{T}^2}\right)_\#(\beta) = \alpha^{r_3} \beta^{r_4}$. Moreover, $\delta$ has the Borsuk-Ulam property with respect to $\tau_j$ if, and only if, $\bar{\delta}$ has the Borsuk-Ulam property with respecto to $\bar{\tau_j}$.
\end{lemma}

\begin{proof}
	The first part of the statement of the Lemma follows from the presentations of $\pi_1(\mathbb{T}^2)$ and $\pi_1(\mathbb{T}^3)$ given in the beginning of the section and~(\ref{diag:f_induced}). Suppose that $\bar{\delta}$ has the Borsuk-Ulam property with respect to $\bar{\tau_j}$. Given $f\colon (\mathbb{T}^3,m_3) \to (\mathbb{T}^3,m_3)$ a representative map of $\delta$, by our assumption, there exists a point $(e^{\theta_1 i}, e^{\theta_2i}) \in \mathbb{T}^2$ such that
	$f|_{\mathbb{T}^2}(\bar{\tau_j}(e^{\theta_1 i}, e^{\theta_2i})) = f|_{\mathbb{T}^2}(e^{\theta_1 i}, e^{\theta_2i})$. So
	$$f(\tau_j(e^{\theta_1 i}, e^{\theta_2i},1)) 
	= (f|_{\mathbb{T}^2}(\bar{\tau_j}(e^{\theta_1 i}, e^{\theta_2i})),1)
	= (f|_{\mathbb{T}^2}(e^{\theta_1 i}, e^{\theta_2i}),1)
	= f(e^{\theta_1 i}, e^{\theta_2i},1).$$ 
	Thus $\delta$ has the Borsuk-Ulam property with respect to $\tau_j$. 
	
	Suppose now that $\bar{\delta}$ does not have the Borsuk-Ulam property with respect to $\bar{\tau_j}$. By~\cite[Theorem~7]{GonGuaLaa} we have the following commutative diagram:
	\begin{equation}\label{diag:T^2}\begin{gathered}\xymatrix{
				\pi_1 (\mathbb{T}^2) \ar[rr]^{\bar{\varphi_j}} \ar[dd]_{(\pi_{\bar{\tau_j}})_\#} \ar@/^0.9cm/[rrrr]^{ \bar{\delta}_\# }  
				&& P_2(\mathbb{T}^2) \ar@{^{(}->}[dd] \ar[rr]^{(p_1)_\#}  && \pi_1(\mathbb{T}^2)  \\
				&    &  & &  \\
				\pi_1(\mathbb{T}^2_{\bar{\tau_j}}) \ar[rr]^{\bar{\psi_j}} \ar[rd]_{\theta_{\bar{\tau_j}}}  && B_2(\mathbb{T}^2) \ar[ld]  && \\
				& \ztwo .& & &
	}\end{gathered}\end{equation}
	By~\cite[Theorem~10 and Remark~11]{GonGuaLaa} there exists pure braids $\lambda,\gamma \in B_2(\mathbb{T}^2)$ which generates the centre of this group and $(p_1)_\#(\lambda) = \alpha$ and $(p_1)_\# (\gamma) = \beta$. Using Remarks~\ref{rem:orbit}  and~\ref{rem:presentation_braids}, we define the following homomorphisms:
	\begin{center}
		$\varphi_j\colon \pi_1(\mathbb{T}^3) \cong \pi_1(\mathbb{T}^2) \oplus \pi_1(\mathbb{S}^1) \longrightarrow P_2(\mathbb{T}^3)_{\mathbb{S}^1} \cong P_2(\mathbb{T}^2) \oplus \pi_1(\mathbb{S}^1) \text{ by }$
		
		$\varphi_j|_{\pi_1(\mathbb{T}^2)} = \bar{\varphi_j}$ and  $\varphi_j(c) = \lambda^{u} \gamma^{v}  c$
	\end{center}
	and
	\begin{center}
		$\psi_j\colon \pi_1(\mathbb{T}^3_{\tau_j}) \cong \pi_1(\mathbb{T}^2_{\bar{\tau_j}}) \oplus \pi_1(\mathbb{S}^1) \longrightarrow B_2(\mathbb{T}^3)_{\mathbb{S}^1} \cong B_2(\mathbb{T}^2) \oplus \pi_1(\mathbb{S}^1) \text{ by }$
		
		$\psi_j|_{\pi_1(\mathbb{T}^2_{\bar{\tau_j}})} = \bar{\psi_j} $ and  $\psi_j(c) = \lambda^{u} \gamma^{v}  c$.
	\end{center}
	We leave to the reader to check that from diagram~(\ref{diag:T^2}) we obtain the following commutative diagram:
	\begin{equation}\label{diag:T^3}\begin{gathered}\xymatrix{
				\pi_1 (\mathbb{T}^3) \ar[rr]^{\varphi_j} \ar[dd]_{(\pi_{\tau_j})_\#} \ar@/^0.9cm/[rrrr]^{ \beta_\# } \ar[dr]^-{p_{\#}} 
				&& P_2(\mathbb{T}^3)_{\mathbb{S}^1} \ar@{^{(}->}[dd] \ar[rr]^{(p_1)_\#} \ar[dl]_-{({p \circ p_{1}})_{\#}} && \pi_1(\mathbb{T}^3)  \\
				&  \pi_{1}(\mathbb{S}^1)  &  & &  \\
				\pi_1(\mathbb{T}^3_{\tau_j}) \ar[rr]^{\psi} \ar[rd]_{\theta_{\tau_j}} \ar[ur]^-{ \overline{p_{\tau_j}}_\# } && B_2(\mathbb{T}^3)_{\mathbb{S}^1} \ar[ld] \ar[ul]_-{\overline{p \circ p_1}_\# } && \\
				& \ztwo .& & &
	}\end{gathered}\end{equation}
	Thus, $\delta$ does not have the Borsuk-Ulam property with respect to $\tau_j$ by Theorem~\ref{th:mainborsuk_braid} and this ends the proof.
\end{proof}

\begin{proof}[Proof of Theorem~\ref{th:tau}]
	The result follows from~(\ref{diag:free_pointed}), Lemma~\ref{lem:t3_t2} and \cite[Theorems~1 and~2]{GonGuaLaa}.
\end{proof}

\section{Appendix} \label{appendix}		

Let $A: \mathbb{T}^{2} \longrightarrow \mathbb{T}^{2} $ be a homeomorphism such that  $[A_{\#}] = { \left[ \begin{array}{cc} a_{11} & a_{12} \\ a_{21} & a_{22} \\ \end{array} \right]}$ and $x_{1} =[(0,0)] \in \mathbb{T}^{2}.$ We denote $x_{2} = [(q, q)]$ for $q $ small. From~\cite[Lemma 6.4]{Vick} we can suppose $A(x_{1}) = x_{1}.$ In this section we will give a presentation for 
$$P_2(MA)_{\mathbb{S}^{1}} = \pi_{1}(MA \times_{\mathbb{S}^{1}} MA-\Delta,(<x_{1},0>,<x_{2},0>)), \,\,\,\, and $$  
$$B_2(MA)_{\mathbb{S}^{1}} = \pi_{1}(\frac{MA \times_{\mathbb{S}^{1}} MA-\Delta}{\tau^{'}}, ([<x_{1},0>,<x_{2},0>])),$$ 
in the case where  $[A_{\#}] = { \left[ \begin{array}{cc} 1 & 1 \\ 0 & 1 \\ \end{array} \right]},$  where $\tau^{'}(<x,t>,<y,t>) = (<y,t>, <x,t>).$ Furthermore we present other useful results.

\begin{theorem}\label{presentation_p2_t2}\cite[Section 4]{FadHus}
	The following is a presentation of $P_2(\mathbb{T}^{2})$:
	
	\noindent generators: $\rho_{1,1}, \rho_{1,2}, \rho_{2,1}, \rho_{2,2}, B$.
	
	\noindent relations:
	
	\begin{enumerate}[(i)]
		\item\label{it:presentation_p2_t2a} $[ \rho_{1,1} , \rho_{1,2}^{-1} ] = [ \rho_{2,1}, \rho_{2,2}^{-1} ] = B$.
		
		\item\label{it:presentation_p2_t2b} $\rho_{2,k} \rho_{1,k} \rho_{2,k}^{-1} = B \rho_{1,k} B^{-1}$  and 
		$\rho_{2,k}^{-1} \rho_{1,k} \rho_{2,k} = \rho_{1,k} [B^{-1}, \rho_{1,k}]$ for all $k \in \{ 1,2 \}$.
		
		\item\label{it:presentation_p2_t2c} $\rho_{2,1} \rho_{1,2} \rho_{2,1}^{-1} = B \rho_{1,2} [ \rho_{1,1}^{-1} , B ]$ and
		$\rho_{2,1}^{-1} \rho_{1,2} \rho_{2,1} = B^{-1} [ B, \rho_{1,1} ] \rho_{1,2} [ B^{-1} , \rho_{1,1} ]$.
		
		\item\label{it:presentation_p2_t2d} $\rho_{2,2} \rho_{1,1} \rho_{2,2}^{-1} = \rho_{1,1} B^{-1}$ and 
		$\rho_{2,2}^{-1} \rho_{1,1} \rho_{2,2} = \rho_{1,1} B [ B^{-1} , \rho_{1,2} ]$.
	\end{enumerate}
	
\end{theorem}	

\begin{theorem}\label{presentation_b2_t2}
	The following is a presentation of $B_2(\mathbb{T}^{2})$:
	
	\noindent generators: $\rho_{1,1}, \rho_{1,2}, \rho_{2,1}, \rho_{2,2}, B, \sigma $.
	
	\noindent relations:
	
	\begin{enumerate}[(i)]
		\item\label{it:presentation_p2_t2a} $[ \rho_{1,1} , \rho_{1,2}^{-1} ] = [ \rho_{2,1}, \rho_{2,2}^{-1} ] = B,$ and $\sigma^{2} = B.$
		
		\item\label{it:presentation_p2_t2b} $\rho_{2,k} \rho_{1,k} \rho_{2,k}^{-1} = B \rho_{1,k} B^{-1}$  and 
		$\rho_{2,k}^{-1} \rho_{1,k} \rho_{2,k} = \rho_{1,k} [B^{-1}, \rho_{1,k}]$ for all $k \in \{ 1,2 \}$.
		
		\item\label{it:presentation_p2_t2c} $\rho_{2,1} \rho_{1,2} \rho_{2,1}^{-1} = B \rho_{1,2} [ \rho_{1,1}^{-1} , B ]$ and
		$\rho_{2,1}^{-1} \rho_{1,2} \rho_{2,1} = B^{-1} [ B, \rho_{1,1} ] \rho_{1,2} [ B^{-1} , \rho_{1,1} ]$.
		
		\item\label{it:presentation_p2_t2d} $\rho_{2,2} \rho_{1,1} \rho_{2,2}^{-1} = \rho_{1,1} B^{-1}$ and 
		$\rho_{2,2}^{-1} \rho_{1,1} \rho_{2,2} = \rho_{1,1} B [ B^{-1} , \rho_{1,2} ]$.
		
		\item $\sigma \rho_{1,k} \sigma^{-1} = \rho_{2,k},$ for all $k \in \{ 1,2 \}$.
		
		\item $ \sigma \rho_{2,k} \sigma^{-1} = B \rho_{1,k} B^{-1}, $ for all $k \in \{ 1,2 \}$.
	\end{enumerate}
	
\end{theorem}

The elements $\rho_{1,1}$ and $\rho_{1,2}$ given in \cite{FadHus} are generators of $\pi_1(\mathbb{T}^{2} - x_2, x_1).$ The inclusion of $\rho_{1,1}$ and $\rho_{1,2}$ in $\mathbb{T}^{2}$ are generators of $\pi_1(\mathbb{T}^{2}, x_1).$  
We consider the presentations $\pi_{1}(\mathbb{T}^{2},x_{1}) = <\rho_{1,1}, \rho_{1,2}| \rho_{1,1}\rho_{1,2}\rho_{1,1}^{-1}\rho_{1,2}^{-1} $ $=1>,$ and 
$\pi_{1}(MA,<x_{1},0>)$ $ = <\tilde{\alpha}, \tilde{\beta},c| \tilde{\alpha} \tilde{\beta} \tilde{\alpha}^{-1} \tilde{\beta}^{-1} = 1, c\tilde{\alpha} c^{-1} = \tilde{\alpha}^{a_{11}}\tilde{\beta}^{a_{21}}, 
c \tilde{\beta} c^{-1} = \tilde{\alpha}^{a_{12}}\tilde{\beta}^{a_{22}}>,$ where $\tilde{\alpha} = <\rho_{1,1},0>,$ $\tilde{\beta} = <\rho_{1,2},0>,$ $c = <x_{1},t>$ 
and the matrix of homeomorphism $A$ given by;
$$ A = { \left[ \begin{array}{cc} a_{11} & a_{12} \\ a_{21} & a_{22} \\ \end{array} \right]} .$$
By abuse of notation we are denoting $A = [A_{\#}].$

Since $\pi_1(\mathbb{S}^{1}) \approx \z $ then the two short exact sequences in Diagram \ref{diag-fiber-bundle-induced-2} split. Thus from \cite{Joh} we obtain;
$ P_2(MA)_{\mathbb{S}^{1}} \approx P_2(\mathbb{T}^{2}) \rtimes \pi_1(\mathbb{S}^{1}) $ and 
$ B_2(MA)_{\mathbb{S}^{1}} \approx B_2(\mathbb{T}^{2}) \rtimes \pi_1(\mathbb{S}^{1}). $ More precisely we have;

\begin{theorem} \label{presentation-P_2(MA)_{S^{1}}} 
	The following is a presentation of $P_2(MA)_{\mathbb{S}^{1}}$:
	
	\noindent generators: $\rho_{1,1}, \rho_{1,2}, \rho_{2,1}, \rho_{2,2}, \tilde{c}, B$
	
	\noindent relations: The relations of $P_{2}(\mathbb{T}^{2}),$ 
	
	\noindent $\tilde{c} \rho_{i,j} \tilde{c}^{-1} = W_{ij}, \,\, i,j =1,2,$ where each $W_{ij},$ $i,j=1,2$ is a word in $\rho_{k,l}, \, k,l \in 
	\{1,2\}.$
	
\end{theorem}	

\begin{theorem} \label{presentation-$B_2(MA)_{S^{1}}$} 
	The following is a presentation of $B_2(MA)_{\mathbb{S}^{1}}$:
	
	\noindent generators: $\rho_{1,1}, \rho_{1,2}, \rho_{2,1}, \rho_{2,2}, \bar{c}, \sigma, B,$
	where $\bar{c} = {p_{\tau^{'}}}(\tilde{c}).$
	
	\noindent relations: The relations of $B_{2}(\mathbb{T}^{2}),$ 
	
	\noindent $\bar{c} \rho_{ij} \bar{c}^{-1} = \overline{W}_{ij}, \,\, i,j =1,2,$ where each $\overline{W}_{ij},$ $i,j=1,2$ is a word in $B_{2}(\mathbb{T}^{2}),$
	
	\noindent $ \bar{c} \sigma \bar{c}^{-1} = \overline{W}(\sigma), $ where $\overline{W}(\sigma)$ is a word in $B_{2}(\mathbb{T}^{2}).$  
	
\end{theorem}

\begin{proposition} \label{presentation-P_2(MA)_{S^{1}}-2} 
	If $A = { \left[ \begin{array}{cc} 1 & 1 \\ 0 & 1 \\ \end{array} \right]}$ then a presentation of $P_2(MA)_{\mathbb{S}^{1}}$ is given by:
	
	\noindent generators: $\rho_{1,1}, \rho_{1,2}, \rho_{2,1}, \rho_{2,2}, \tilde{c}, B$
	
	\noindent relations: The relations of $P_{2}(\mathbb{T}^{2})$ and
	
	\item $\tilde{c} \rho_{k,1} \tilde{c}^{-1} = \rho_{k,1},$ \,\,\,\,\,\, $\tilde{c} \rho_{k,2} \tilde{c}^{-1} = \rho_{k,1}^{-1} \rho_{k,2}, 
	\,\,\,\, k =1,2.$
\end{proposition}
\begin{proof}
	The inclusion $ MA \, \smallsetminus <[(0,0)]\times I> $ in the first coordinate of $MA \times_{\mathbb{S}^{1}} MA-\Delta,(<x_{1},0>,<x_{2},0>)$ is a fiberwise map over $\mathbb{S}^1$. Therefore the calculation of the elements   $\tilde{c} \rho_{11} \tilde{c}^{-1},  \tilde{c} \rho_{12} \tilde{c}^{-1}$ is exactly the ones from the presentation of 
	$\pi_1(MA \, \smallsetminus <[(0,0)] \times I>).$ So the result follows from  \cite[Theorem 2.2, Case III]{GonPenVie04}.
\end{proof}

\begin{proposition} \label{presentation-$B_2(MA)_{S^{1}}$-2}
	If $A = { \left[ \begin{array}{cc} 1 & 1 \\ 0 & 1 \\ \end{array} \right]}$ then a presentation of $B_2(MA)_{\mathbb{S}^{1}}$ is 
	given by:
	
	\noindent generators: $\rho_{1,1}, \rho_{1,2}, \rho_{2,1}, \rho_{2,2}, \bar{c}, \sigma, B$
	
	\noindent relations: The relations of $P_2(MA)_{\mathbb{S}^{1}},$ 
	
	\noindent $\bar{c} \sigma \bar{c}^{-1} = \sigma.$ 
\end{proposition}
\begin{proof}
	Since $({p_{\tau^{'}}})(\rho_{ij}) = \rho_{ij}$ then the relations $\tilde{c} \rho_{ij} \tilde{c}^{-1}$ and $\bar{c} \rho_{ij} \bar{c}^{-1}$ are equal. The conjugation $\bar{c} \sigma \bar{c}^{-1}$ follows from the theorem below.
\end{proof}

\begin{theorem} \label{conjugation-sigma}
	Let $\sigma$ be the element on $B_{2}(\mathbb{T}^{2}),$ the full 2-string braid group of $\mathbb{T}^{2},$ as defined in \cite{FadHus} and $A:\mathbb{T}^{2} \longrightarrow \mathbb{T}^{2}$ a homeomorphism of $\mathbb{T}^{2}$ whose the matrix of $A_{\#}$ we are representing by the same symbol. If $\bar{c}$ is the element in 
	$B_{2}(MA)_{S^{1}},$ as defined on Theorem \ref{presentation-$B_2(MA)_{S^{1}}$}, then 
	$$ \bar{c} \sigma \bar{c}^{-1} = \sigma^{det(A)} $$
	where $det(A) = \pm 1 $ is the determinant of $A.$
\end{theorem}
\begin{proof}
	We know that $A(x_{1})=x_{1}.$ Since $A$ is a homeomorphism and $A(x_{1}) = x_{1}$ there exists a neighborhood $V$ of $x_{1}$ and a small 
	disc $D \subset V$ such that $A(D) \subset V. $ Also, we can suppose that $x_{2} \in D.$ Thus in $V$ we can define the local degree of $A$, 
	see \cite{Sko}. Since $A$ is a homeomorphism and $det(A) = \pm 1$ then the degree of $A$ is 1 if $det(A) = 1$ or -1 if $det(A) = -1.$ 
	Therefore, we can suppose in the neighborhood $V,$ that $A$ is the identity or it is a reflection around an axis passing through the 
	center of $D.$ In particular we can choose $x_{2}$ in this axis.  
	
	In $MA$ we have $(x,0) \sim (A(x),1).$ Thus in $MA \times_{S^{1}} MA-\Delta$ we have $((x,0),(y,0)) \sim ((A(x),1),$ $(A(y),1))$ and 
	therefore in $\frac{MA \times_{S^{1}} MA-\Delta}{\tau^{'}},$ we have $[(x,0),(y,0)] \sim [(A(x),1),(A(y),1)].$ 
	We know also that $\tilde{c} \rho \tilde{c}^{-1} = (A \times A)_{\#}(\rho).$ Firstly, suppose $A$ the identity in $V.$ 
	Consider $\sigma $ in $D.$ Since $(x,0) \sim (A(x),1)$ and $A$ is the identity then $\bar{c} \sigma \bar{c}^{-1} = \sigma$ in $B_{2}(MA)_{S^{1}}$ 
	because the path $\bar{c} \sigma \bar{c}^{-1}$ is homotopic to $\sigma.$ Now, if $A$ is a reflection in $V$ by  an analogous argument we prove that $\bar{c} \sigma \bar{c}^{-1} = \sigma^{-1}.$	
\end{proof}

We denote  $$x = \rho_{2,1}, \,\,\, y = \rho_{2,2}, \,\,\, w = \rho_{1,1}B^{-1}\rho_{2,1}, \,\,\, z = \rho_{12}B^{-1}\rho_{22}$$   and $F(x,y)$ the free group in the two generators $x$ and $y.$ Identifying $\z [w] \oplus \z [z]$ with $\mathbb{Z} \oplus \mathbb{Z} $ then from \cite[Theorem 12]{GonGuaLaa} we obtain;
$$P_2 (\mathbb{T}^{2}) \approx F( x , y ) \oplus \mathbb{Z} \oplus \mathbb{Z}. $$

\begin{proposition} \label{prop-conjugations}
	If $A = { \left[ \begin{array}{cc} 1 & 1 \\ 0 & 1 \\ \end{array} \right]}$ then we have $P_{2}(MA)_{\mathbb{S}^{1}} \cong (F(x,y) \oplus F(w) \oplus F(z)) \rtimes F(\tilde{c}).$ Furthermore in $P_{2}(MA)_{\mathbb{S}^{1}}$ and $B_{2}(MA)_{\mathbb{S}^{1}}$ we have the following relations;
	$$\tilde{c} x \tilde{c}^{-1} = x , \,\, \tilde{c} y \tilde{c}^{-1} = x^{-1}y, \,\, \tilde{c} w \tilde{c}^{-1} = w, \,\, 
	\tilde{c} z \tilde{c}^{-1} = w^{-1}z, \,\, \tilde{c} B \tilde{c}^{-1} = B.$$  		
	$$\bar{c} \sigma \bar{c}^{-1} = \sigma, \,\, \sigma x \sigma^{-1} = B x^{-1} w, \,\, \sigma y \sigma^{-1} = B y^{-1} z, \,\, 
	\sigma w \sigma^{-1} =  w, \,\, \sigma z \sigma^{-1} = z.$$ 	
\end{proposition}
\begin{proof}
	The proof follows easily from Propositions \ref{presentation-P_2(MA)_{S^{1}}-2} and  \ref{presentation-$B_2(MA)_{S^{1}}$-2}.
\end{proof}

Now we will present some useful results.

\begin{lemma} \label{lemma-gamma}
	Let $\gamma: F(x,y) \rightarrow \mathbb{Z} \oplus \mathbb{Z}$ be the homomorphism defined by $ \gamma(x) = (1,0)$ and $\gamma(y) = (0,1).$ We have;
	
	\begin{enumerate}[(1)]
		
		\item The kernel of $\gamma$ is the normal closure subgroup generated by $B = [x,y^{-1}].$
		
		\item If $\gamma(R) = (0,0)$ then $R = \displaystyle \prod_{i=1} (x^{e_{i}} y^{f_{i}} B^{t_{i}} y^{-f_{i}} x^{-e_{i}}),$ 
		for some $e_{i},f_{i},t_{i} $ $\in \mathbb{Z}.$
		
		\item Consider the composition of homomorphisms 
		$P_{2}(\mathbb{T}^{2}) \stackrel{\eta}{\longrightarrow} F(x,y)  \stackrel{\gamma}{\longrightarrow} \mathbb{Z} \oplus \mathbb{Z},$ where 
		$\eta$ is the projection. We have $\gamma(\eta(P_{i})) = (a_{i}, b_{i})$ for some $a_{i},b_{i} \in \mathbb{Z}.$ 
		Since $\gamma(x^{a_{i}}y^{b_{i}}) = (a_{i}, b_{i}) =  \gamma(\eta(P_{i}))$ then we can write;
		\begin{equation} \left \{\begin{array}{l}
				P_{1} = (x^{a_{1}}y^{b_{1}}A_{1}; m_{1}, n_{1}), \\
				P_{2} = (x^{a_{2}}y^{b_{2}}A_{2}; m_{2}, n_{2}), \\ 
				P_{3} = (x^{a_{3}}y^{b_{3}}A_{3}; m_{3}, n_{3}), \\
			\end{array} \right.
		\end{equation}
		where $A_{1} = \displaystyle \prod_{i=1} (x^{e_{i}} y^{f_{i}} B^{r_{i}} 
		y^{-f_{i}} x^{-e_{i}}),$ $A_{2} = \displaystyle \prod_{i=1} (x^{h_{i}} y^{j_{i}} B^{s_{i}} 
		y^{-j_{i}} x^{-h_{i}})$ and 
		$A_{3} = \displaystyle \prod_{i=1} (x^{k_{i}} y^{l_{i}} B^{t_{i}} 
		y^{-l_{i}} x^{-k_{i}}).$
	\end{enumerate}
\end{lemma}
\begin{proof}
	The items $(2)$ and $(3)$ follows from definition of $\gamma.$ The item $(1)$ follows from \cite[Page 650]{Lyn-50}.
\end{proof}	


\begin{lemma} \label{lemma-epsilon}
	Let  $ \mathcal{E}: Kernel(\gamma) = <B>   \rightarrow \mathbb{Z}$ be the 
	homomorphism defined by $\mathcal{E}(B) = 1.$ We have;
	
	\begin{enumerate}[(1)]
		
		\item $\mathcal{E}(p(x,y)B^{l}p(x,y)^{-1}) = l.$
		
		\item $\mathcal{E}([x^{n}, y^{k}]) = -nk,$ for all $n,k \in \mathbb{Z}.$
		
		\item If $W_{n,d} = (x^{-n}y)^{d} x^{nd} y^{-d}$  then 
		$\mathcal{E}(W_{n,d}) = \displaystyle \frac{-d(d+1)n}{2}$ for all $n, d \in \mathbb{Z}.$
		
		\item $\mathcal{E}((Bx^{-1})^{n} x^{n}) = n,$ for all $n \in \mathbb{Z}.$

	\end{enumerate}
\end{lemma}
\begin{proof}
	$(1)$ The proof of this item is trivial.
	$(2)$ We have $B = [x, y^{-1}].$ Note that 
	$$
	\begin{array}{l}
		[x,y]  =  y (y^{-1}xyx^{-1})y^{-1} = y [y^{-1}, x] y^{-1} = y B^{-1}y^{-1}, \\
		
		[x^{-1}, y]  =  y x^{-1}(xy^{-1}x^{-1}y)x y^{-1}  = y x^{-1}B x y^{-1}, \\
		
		[x^{-1}, y^{-1}] = x^{-1} ( y^{-1} x y x^{-1}) x = x^{-1} B^{-1} x. \\
		
	\end{array}
	$$
	
	Given $n > 0$ we have;
	\begin{equation} \label{ep-1}
		\begin{array}{l}
			[x^{n}, y] = x(x^{n-1}yx^{-n+1} y^{-1}) x^{-1} (xyx^{-1}y^{-1}) = x[x^{n-1}, y]x^{-1} [x,y],\\
			
			[x^{-n}, y] = x^{-1}(x^{-n+1} y x^{n-1} y^{-1} ) x (x^{-1}yxy^{-1}) = x^{-1}[x^{-(n-1)}, y]x [x^{-1}, y].\\
			
		\end{array}
	\end{equation}

	Therefore, by induction we can prove the following:
	\begin{equation} \label{ep-2}
		\begin{array}{l}
			[x^{n}, y] = \displaystyle \prod_{j=1}^{n} (x^{n-j}[x,y] x^{-n+j})  = \prod_{j=1}^{n} (x^{n-j}y B^{-1} y^{-1} x^{-n+j}), \\
			
			[x^{-n}, y] = \displaystyle \prod_{j=1}^{n} (x^{-n+j}[x^{-1},y] x^{n-j})  = \prod_{j=1}^{n} (x^{-n+j}yx^{-1} B x y^{-1} x^{n-j}). \\
		\end{array}
	\end{equation}
	
	Now, given $k > 0,$ using the same argument as in \eqref{ep-1} and induction we can prove;
	$$
	\begin{array}{l}
		[x^{n}, y^{k}] = \displaystyle \prod_{j=0}^{k-1} y^{j}[x^{n}, y]y^{-j}, \\
		
		[x^{n}, y^{-k}] = \displaystyle \prod_{j=0}^{k-1} y^{-j}[x^{n}, y^{-1}]y^{j}. \\
		
	\end{array}
	$$ 
	
	By \eqref{ep-2} we obtain;
	\begin{equation} \label{ep-3}
		\begin{array}{l}
			[x^{n}, y^{k}] = \displaystyle \prod_{j=0}^{k-1} \left[ \prod_{l=1}^{n} (y^{j}x^{n-l}y B^{-1} y^{-1} x^{-n+l}y^{-j}) \right], \\
			
			[x^{n}, y^{-k}] = \displaystyle \prod_{j=0}^{k-1} \left[ \prod_{l=1}^{n} (y^{-j}x^{-n+l}yx^{-1} B x y^{-1} x^{n-l}y^{j}) \right]. \\
		\end{array}
	\end{equation} 
	
	Applying the homomorphism $\mathcal{E}$ in the equations \eqref{ep-3} we obtain 
	$$
	\mathcal{E}([x^{n}, y^{k}]) = -nk \,\,\,\,\,\, and \,\,\,\,\,\, \mathcal{E}([x^{n}, y^{-k}]) = nk.
	$$
	
	The prove for the cases $[x^{-n}, y^{k}]$ and $[x^{-n}, y^{-k}]$ are made in the analogous way.
	
	\
	
	(3) Note that $W_{n,1} = x^{-n}y x^{n} y^{-1} = 
	[x^{-n}, y].$ Thus $\mathcal{E}(W_{n,1}) = -n$ by item $(2).$ 
	We take $d > 0.$ In this situation we have:
	$$W_{n,d+1} = W_{n,d}(y^{d} x^{-nd}[x^{-n},y] x^{nd} y^{-d}) (y^{d}[x^{-nd},y] y^{-d}).$$
	
	By induction and item $(2)$ we have $\mathcal{E}(W_{n,d+1}) = $ $\mathcal{E}(W_{n,d}) -n -nd = $ 
	$\frac{-d(d+1)n}{2} - \frac{2n(d+1)}{2} = $
	$ \frac{-(d+1)(d+2)n}{2}.$
	
	Note that 
	$$W_{n,-d} = x^{nd} y^{-d} (W_{n,d})^{-1} y^{d} x^{-nd} [x^{nd}, y^{-d}]. $$
	Thus, $\mathcal{E}(W_{n,-d}) = - \mathcal{E}(W_{n,d})$
	$ -nd^{2} = $ $\frac{d(d+1)n}{2} -nd^{2} = $ $\frac{-d(d-1)n}{2},$ and therefore the result follows.
	
	\
	
	(4) The proof of this item follows from the following relations that are easily proved using induction.
	$$ \begin{array}{l}
		\displaystyle (Bx^{-1})^{n} x^{n} = \prod_{j=0}^{n-1} (x^{-j} B x^{j}),  \,\,\,\,\,\, and  \\
		\displaystyle (Bx^{-1})^{-n} x^{-n} = \prod_{j=1}^{n} (x^{j} B^{-1} x^{-j})
	\end{array}
	$$
	for all $n > 0.$
	
\end{proof}


\subsection*{Acknowledgements}

This work was supported by Universidade Estadual de Santa Cruz (UESC), Brazil, Grant no SEI: 073.6766.2019.0020903-60 and Projeto Tem\'atico FAPESP, grant no 2016/24707-4: Topologia Alg\'ebrica, Geom\'etrica e Diferencial.

The first author would like to thank Prof. Peter Pavesic
for helpful  discussion and useful information relative Lemma \ref{lem:quo_fibration}.


\end{document}